\documentclass{amsart}%
\usepackage{amsfonts}
\usepackage{amsmath}
\usepackage{amssymb}
\usepackage{graphicx}
\usepackage{xcolor}%
\providecommand{\U}[1]{\protect\rule{.1in}{.1in}}
\newtheorem{theorem}{Theorem}[section]
\newtheorem{lemma}[theorem]{Lemma}
\theoremstyle{definition}
\newtheorem{definition}[theorem]{Definition}
\newtheorem{example}[theorem]{Example}

\theoremstyle{remark}
\newtheorem{remark}[theorem]{Remark}
\theoremstyle{plain}

\newtheorem{corollary}[theorem]{Corollary}
\newtheorem{notation}[theorem]{Notation}
\newtheorem{proposition}[theorem]{Proposition}
\numberwithin{equation}{section}
\begin{document}
\title[Higher-order homogeneous hypoelliptic operators]{Fundamental solution for higher-order homogeneous hypoelliptic operators
structured on H\"{o}rmander vector fields}
\author{Stefano Biagi}
\address[S. Biagi and M. Bramanti]{Dipartimento di Matematica. Politecnico di Milano.\\
\indent Via Bonardi 9. 20133 Milano. Italy.}
\email[S.\,Biagi]{stefano.biagi@polimi.it}
\author{Marco Bramanti}
\email[M.\,Bramanti]{marco.bramanti@polimi.it}
\date{\today}

\begin{abstract}
We introduce and study a new class of higher-order differential operators
defined on $\mathbb{R}^{n}$, which are built with H\"{o}rmander vector fields,
homogeneous w.r.t.\,a family of dilations (but not left-invariant w.r.t.\,any
structure of Lie group) and have a structure such that a suitably lifted
version of the operator is hypoelliptic. We call these operators
\textquotedblleft generalized Rockland operators\textquotedblright. We prove
that these operators are themselves hypoelliptic and, under a natural
condition on the homogeneity degree, possess a global fundamental solution
$\Gamma\left(  x,y\right)  $ which is jointly homogeneous in $\left(
x,y\right)  $ and satisfies sharp pointwise estimates. Our theory can be
applied also to some higher-order heat-type operators and their fundamental solutions.

\end{abstract}
\maketitle

\section{Introduction}

\label{sec:Intro}

\noindent\textbf{Second order hypoelliptic operators} \vspace{0.1cm}

\noindent The theory of second order hypoelliptic operators is deeply related
to the notion of system of H\"{o}rmander vector fields, in view of the famous
hypoellipticity theorem proved by H\"{o}rmander in \cite{Hor}. Assume that%
\[
X_{0},X_{1},...,X_{m}%
\]
is a family of smooth real vector fields defined in some open set
$\Omega\subset\mathbb{R}^{n}$ such that the Lie algebra generated by these
$X_{i}$ at every point of $\Omega$ has rank $n$. This means that if we define
the \emph{commutator} of two vector fields $X,Y$ letting, as usual,%
\[
\left[  X,Y\right]  =XY-YX,
\]
then, among the vector fields $X_{i}$ and those obtained taking their
commutators $\left[  X_{i},X_{j}\right]  $, the iterated commutators $\left[
\left[  X_{i},X_{j}\right]  ,X_{k}\right]  $, and so on up to a certain step,
at every point of $\Omega$ we can find $n$ independent vectors. In this case,
we say that $X_{0},X_{1},...,X_{m}$ are \emph{a system of H\"{o}rmander vector
fields}, or that they satisfy \emph{H\"{o}rmander's condition} in $\Omega$.
Then H\"{o}rmander's theorem says that the operator%
\begin{equation}
L=\sum_{i=1}^{m}X_{i}^{2}+X_{0}, \label{Hor}%
\end{equation}
which under the above assumptions is called a \emph{H\"{o}rmander operator},
is \emph{hypoelliptic} in $\Omega$, that is for every distribution
$u\in\mathcal{D}^{\prime}\left(  \Omega\right)  $ and every open subset
$\Omega^{\prime}\subset\Omega$, if $Lu\in C^{\infty}\left(  \Omega^{\prime
}\right)  $ then $u\in C^{\infty}\left(  \Omega^{\prime}\right)  $.

While many \emph{local }results have been proved for general H\"{o}rmander
operators (\ref{Hor}) (see e.g. \cite{Bo}, \cite{RS}, \cite{NSW}, \cite{J},
\cite{SC}, \cite{JSC}), \emph{global} results have been established only in
special cases, in particular when $X_{0},X_{1},...,X_{m}$ are left-invariant
in $\mathbb{R}^{n}$, with respect to a Lie group \textquotedblleft
translation\textquotedblright, and the operator $L$ is $2$-homogeneous with
respect to a family of (diagonal, nonisotropic) dilations. This is the context
of H\"{o}rmander operators on homogeneous groups \textquotedblleft stratified
of type II\textquotedblright, in the language of Rothschild-Stein \cite{RS},
or, when $X_{0}$ is lacking, the simpler case of sublaplacians on stratified
(Carnot) groups (precise definitions will be given in Section
\ref{sec left-invariant}). In this situation, Folland \cite{Fo2} has shown the
existence and good properties of a global homogeneous fundamental solution,
and has shown regularity results in the scale of Sobolev spaces adapted to the
vector fields. See also \cite[Chap.\,8]{BBbook} for a complete proof of
\emph{global} regularity results in the scale of both Sobolev and H\"{o}lder spaces.

In the more general situation when $X_{0},X_{1},...,X_{m}$ are a set of
H\"{o}rmander vector fields in $\mathbb{R}^{n}$, homogeneous w.r.t.\,a family
of dilations ($X_{0}$ is $2$-homo\-ge\-ne\-ous while $X_{1},...,X_{m}$ are
$1$-homogeneous), but not left-invariant w.r.t.\,any Lie group operation, a
global fundamental solution with good properties has been constructed for
operators (\ref{Hor}) by Biagi-Bonfiglioli in \cite{BB-lift} and studied in
\cite{BBB-fundsol}, while global estimates in Sobolev spaces have been proved
in \cite{BBB-Sob}. \bigskip

\noindent\textbf{Higher-order operators on homogeneous groups} \vspace{0.1cm}

\noindent For differential operators of order greater than $2$, with real
variable coefficients, a simple powerful sufficient condition for
hypoellipticity, like H\"{o}rmander's condition,\ does not exist, in general.
A more expressive picture can be drawn in the special context of
\emph{homogeneous groups}. A theorem by Helffer-Nourrigat \cite{HN},
culminating a research started in the context of Heisenberg groups with the
work of Rockland \cite{Ro}, and extended to the context of general homogeneous
groups by Beals \cite{Be}, states that if a differential operator $L$ is
\emph{left-invariant }and \emph{homogeneous }of some positive degree on a
homogeneous group $\mathbb{G}$ (precise definitions will be recalled later,
see Section \ref{sec left-invariant}), then $L$ is hypoelliptic if and only if
it satisfies the \emph{Rockland condition} (and in this case we say that $L$
is a \emph{Rockland operator}). The exact formulation of the Rockland
condition will not be recalled here (see \cite[Par.\,4.1.1]{FR}). Instead,
throughout this paper we will use this notion according to the following
definition, which is logically equivalent to the standard one in view of
Helffer-Nourrigat's theorem but avoids the language of representation theory,
which is involved in the statement of Rockland's condition:

\begin{definition}
[Rockland operator]\label{Def Rockland}Given a homogeneous group $\mathbb{G}$,
a \emph{Ro\-ck\-land operator} on $\mathbb{G}$ is a left-invariant
hypoelliptic differential operator $L$, homogeneous of some positive degree.
\end{definition}

An explicit class of higher-order Rockland operators can be described as
follows. Let $X_{1},...,X_{m}$ be a family of left-invariant H\"{o}rmander
vector fields on a homogeneous group $\mathbb{G}$ such that each $X_{i}$ is
homogeneous of some positive integer degree $\nu_{i}$, and let $\nu_{0}$ be
any common integer multiple of $\nu_{1},...,\nu_{m}$. Then the left-invariant,
$2\nu_{0}$ homogeneous operator%
\begin{equation}
L=\sum_{j=1}^{m}\left(  -1\right)  ^{\frac{\nu_{0}}{\nu_{j}}}X_{j}^{\frac
{2\nu_{0}}{\nu_{j}}} \label{Rock}%
\end{equation}
satisfies the Rockland condition (see \cite[Lem.\,4.1.8]{FR}), in particular
it is hypoelliptic. Therefore $L$ is a Rockland operator. Note that the
operator $L$ is a differential operator of even order, generally greater than
$2$ (as soon as some integer quotient $\nu_{0}/\nu_{j}$ is greater than $1$).

We can say that Rockland operators of the form (\ref{Rock}) generalize to
general homogeneous groups the family of sublaplacians on Carnot groups. Note,
instead, that H\"{o}rmander operators (\ref{Hor}) are not a special case of
(\ref{Rock}), unless the \emph{drift }$X_{0}$ is lacking. Nevertheless, if
(\ref{Hor}) is $2$-homogeneous and left-invariant on a homogeneous group, then
it is a Rockland operator. Note that any composition of two Rockland operators
is still a Rockland operator (because it is still left-invariant, homogeneous,
and hypoelliptic). For instance, any \emph{power} of a left-invariant
homogeneous operator of type (\ref{Hor}) or (\ref{Rock}) is a Rockland
operator. Here we have tacitly exploited a simple fact which is worthwhile to
be pointed out explicitly:

\begin{remark}
\label{Remark composition hypoelliptic}The composition of two operators $L,M$,
which are hypoelliptic in $\Omega\subseteq\mathbb{R}^{n}$, is still
hypoelliptic. Namely, if $u\in\mathcal{D}^{\prime}\left(  \Omega\right)  $ and
for some open $\Omega^{\prime}\subseteq\Omega$ we know that $L\left(
Mu\right)  \in C^{\infty}\left(  \Omega^{\prime}\right)  ,$ then $Mu\in
C^{\infty}\left(  \Omega^{\prime}\right)  $ because $L$ is hypoelliptic, and
then $u\in C^{\infty}\left(  \Omega^{\prime}\right)  $ because $M$ is hypoelliptic.
\end{remark}

\bigskip

\noindent\textbf{Our setting and main results} \vspace{0.1cm}

\noindent In this paper we are interested in generalizing the previous picture
to higher-order operators structured on H\"{o}rmander vector fields in
$\mathbb{R}^{n}$ which are homogeneous w.r.t.\,a family of dilations,
\emph{but not left-invariant} w.r.t.\,any group structure. Our setting is the
following. \bigskip

\noindent\textbf{Assumption (H1)} We assume\ that:

(a) $X_{1},...,X_{m}$ is a family of linearly independent\footnote{Here and
throughout the paper, \emph{linear independence} is meant in the
infinite-dimensional space of the smooth vector fields on $\mathbb{R}^{n}$.},
real smooth vector fields in $\mathbb{R}^{n}$, satisfying H\"{o}rmander's
condition at the origin.

(b) $\mathbb{R}^{n}$ is endowed with a family of \emph{dilations}%
\begin{equation}
\delta_{\lambda}\left(  x\right)  =\left(  \lambda^{\sigma_{1}}x_{1}%
,\lambda^{\sigma_{2}}x_{2},...,\lambda^{\sigma_{n}}x_{n}\right)  ,
\label{delta lambda}%
\end{equation}
for positive integers%
\begin{equation}
1=\sigma_{1}\leq\sigma_{2}\leq...\leq\sigma_{n}, \label{homogeneities_1}%
\end{equation}
and each $X_{i}$ is $\delta_{\lambda}$-\emph{homogeneous of positive integer
degree} $\nu_{i}$, that is%
\begin{equation}
X_{i}(u\circ\delta_{\lambda})=\lambda^{\nu_{i}}(X_{i}u)\circ\delta_{\lambda}
\label{homogeneous op}%
\end{equation}
for every $\lambda>0$ and $u\in C^{\infty}\left( \mathbb{R}^{n}\right)  $.
Moreover%
\begin{equation}
1\leq\nu_{1}\leq\nu_{2}\leq...\leq\nu_{m}. \label{homogeneities2}%
\end{equation}

The number%
\begin{equation}
q=\sum_{i=1}^{n}\sigma_{i} \label{q}%
\end{equation}
will be called \emph{homogeneous dimension}.

\begin{remark}
\label{rem:HormanderRn}By the homogeneity assumptions on $X_{1},...,X_{m}$,
H\"{o}rmander's condition actually holds at every point of $\mathbb{R}^{n}$
(see \cite[Rem.\,3.2.]{BBB-fundsol}).
\end{remark}

\begin{example}
\label{Ex H1}The following examples exhibit families of vector fields
satisfying Assumption (H1). We stress the fact that in all these examples
there cannot exist any Lie group structure in $\mathbb{R}^{n}$ making these
vector fields left-invariant. This happens because one of the vector fields
vanishes at the origin without being identically zero.

(1).\thinspace\thinspace In $\mathbb{R}^{2}$, let
\begin{align*}
X_{1}  &  =\partial_{x_{1}}\\
X_{2}  &  =x_{1}^{k}\,\partial_{x_{2}}%
\end{align*}
with $k=1,2,3...$and let%
\[
\delta_{\lambda}\left(  x_{1},x_{2}\right)  =(\lambda x_{1},\lambda^{k+h}%
x_{2})
\]
with $h=1,2,3...$ Then $X_{1},X_{2}$ satisfy H\"{o}rmander's condition at $0$
and are $\delta_{\lambda}$-homogeneous with
\begin{align*}
\nu_{1}  &  =1\\
\nu_{2}  &  =h.
\end{align*}

(2).\thinspace\thinspace In $\mathbb{R}^{n}$, let
\begin{align*}
X_{1}  &  =\partial_{x_{1}}\\
X_{2}  &  =x_{1}\partial_{x_{2}}+x_{2}\partial_{x_{3}}+\ldots+x_{n-1}%
\partial_{x_{n}}%
\end{align*}
and let%
\[
\delta_{\lambda}(x)=(\lambda x_{1},\lambda^{2}x_{2},\cdots,\lambda^{n}x_{n})
\]
Then $X_{1},X_{2}$ satisfy H\"{o}rmander's condition and are $\delta_{\lambda
}$-homogeneous with
\begin{align*}
\nu_{1}  &  =1\\
\nu_{2}  &  =1.
\end{align*}

(3).\thinspace\thinspace\ In $\mathbb{R}^{3}$, let%
\begin{align*}
X_{1}  &  =\partial_{x_{1}}\\
X_{2}  &  =x_{1}\partial_{x_{2}}+x_{2}^{2}\partial_{x_{3}}%
\end{align*}
and
\[
\delta_{\lambda}(x)=(\lambda x_{1},\lambda^{1+k}x_{2},\lambda^{2+3k}x_{3})
\]
with $k=1,2,3...$ Then $X_{1},X_{2}$ satisfy H\"{o}rmander's condition and are
$\delta_{\lambda}$-homoge\-ne\-ous with
\begin{align*}
\nu_{1}  &  =1\\
\nu_{2}  &  =k.
\end{align*}

\end{example}

Let us introduce some notation which will be used throughout the paper. For
any multi-index%
\[
I=\left(  i_{1},i_{2},...,i_{k}\right)
\]
with $i_{j}\in\left\{  1,2,...,m\right\}  $, we set%
\[
X_{I}=X_{i_{1}}X_{i_{2}}...X_{i_{k}}.
\]
We also define the \emph{length }of the multi-index,
\[
\ell\left(  I\right)  =k,
\]
and the \emph{weight} of the multi-index,%
\begin{equation}
\left\vert I\right\vert =\sum_{h=1}^{k}\nu_{i_{h}}, \label{weight}%
\end{equation}
where $\nu_{i}$ is the homogeneity degree of $X_{i}$, see
(\ref{homogeneities2}).

We are going to study a class of $\nu$-homogeneous operators of the kind%
\begin{equation}
\mathcal{L}=\sum_{|I|=\nu}c_{I}X_{I} \label{general operator L}%
\end{equation}
where $\nu$ is some positive integer and $c_{I}$ are real constants, according
to the following:

\begin{definition}
[Generalized Rockland operators, informal definition]%
\label{Def gen Rockland 1} We say that an operator $\mathcal{L}$ of kind
(\ref{general operator L}) is a \emph{generalized Rockland operator }if the
$X_{i}$ satisfy Assumption (H1), and the constants $\left\{  c_{I}\right\}
_{\left\vert I\right\vert =\nu}$ give to $\mathcal{L}$ a structure such that,
if the $X_{i}$ were \emph{also} left-invariant w.r.t.\,a Lie group structure
in $\mathbb{R}^{n}$ such that the dilations $\delta_{\lambda}$ are group
automorphisms, then $\mathcal{L}$ would be hypoelliptic.
\end{definition}

\medskip

Later (see Section \ref{sec lifting}), we will give a more formal, and
actually more general, formulation of the hypoellipticity condition contained
in Definition \ref{Def gen Rockland 1}, involving a suitable notion of
\emph{lifting }of vector fields. For the moment, the above formulation is
enough to understand the following examples and the meaning of our main results.

\begin{example}
\label{Ex H2}Classes of examples of generalized Rockland operators are the following.

(1). Let $X_{1},...,X_{m}$ be a family of vector fields satisfying Assumption
(H1) (for instance, as in Example \ref{Ex H1}) and (with the same notation of
Assumption (H1)) for any positive integer $\nu_{0}$, common multiple of
$\nu_{1},\nu_{2},...,\nu_{m}$, and any positive integer $k$, let us consider
the operators%
\begin{equation}
\mathcal{L}_{\nu_{0}}^{k}=\left(  \sum_{j=1}^{m}\left(  -1\right)  ^{\frac
{\nu_{0}}{\nu_{j}}}X_{j}^{\frac{2\nu_{0}}{\nu_{j}}}\right)  ^{k}.
\label{power of Rockland}%
\end{equation}
If the $X_{j}$ were also left-invariant on a homogeneous Lie group, this
operator would be the $k$-th power of a Rockland operator of type
(\ref{Rock}), hence it would be hypoelliptic (see Remark
\ref{Remark composition hypoelliptic}). Therefore\ the conditions in
Definition \ref{Def gen Rockland 1} are satisfied.

(1') Special cases of the above operators (\ref{power of Rockland}) are powers
of $2$-homogeneous sublaplacians%
\begin{equation}
\Delta^{k}=\left(  \sum_{j=1}^{m}X_{j}^{2}\right)  ^{k}
\label{power of sublaplacian}%
\end{equation}
or also, assuming all the vector fields $X_{1},...,X_{m}$ $1$-homogeneous,
operators of the kind%
\begin{equation}
\mathcal{L}_{\nu_{0}}=\sum_{j=1}^{m}X_{j}^{2\nu_{0}} \label{sum of even}%
\end{equation}
for some $\nu_{0}=2,3,4...$

(2) Operators of the kind%
\begin{equation}
\mathcal{L}^{k}=\left(  \sum_{j=1}^{m}X_{j}^{2}+X_{0}\right)  ^{k}
\label{power of Hormander}%
\end{equation}
when $X_{0},X_{1},...,X_{m}$ are a system of Hormander vector fields, $X_{0}$
is $2$-ho\-mo\-ge\-ne\-ous and $X_{1},...,X_{m}$ are $1$-homogeneous do not
fit in the class (\ref{power of Rockland}) but they obviously are generalized
Rockland operators because they are the $k$-th power of a (hypoelliptic)
H\"{o}rmander operator.
\end{example}

\begin{remark}
Note that while the hypoellipticity of operators (\ref{power of sublaplacian})
and (\ref{power of Hormander}) is obvious because they are powers of
hypoelliptic operators (by H\"{o}rmander's theorem), the hypoellipticity of
operators (\ref{power of Rockland}) or, in particular, (\ref{sum of even}) is
not obvious when the $X_{i}$ are not left-invariant, but will be actually
proved (see Thm \ref{Thm main hypoellipticity}).
\end{remark}

Before stating our main results, we need to recall the following:

\begin{definition}
\label{Def formal transposed}Given a differential operator $L$ on $\Omega$
(with real smooth coefficients), we will denote by $L^{\ast}$ its formal
\emph{transpose}, defined by%
\begin{equation}
\int_{\Omega}\left(  Lu\right)  \cdot v=\int u\cdot\left(  L^{\ast}v\right)
\text{ for every }u,v\in C_{0}^{\infty}\left(  \Omega\right)  .
\label{transposed}%
\end{equation}

\end{definition}

\begin{remark}
\label{remark trasposti notevoli}It is well known (and will be justified
later, see Remark \ref{rem:ConseguenzeH1H2}) that under Assumption (H1), one
has%
\[
X_{i}^{\ast}=-X_{i}.
\]
Therefore, for instance, for operators of kind (\ref{power of Rockland}), we
simply have%
\[
L^{\ast}=L,
\]
while for operators (\ref{power of Hormander}) one has%
\[
\left(  \left(  \sum_{j=1}^{m}X_{j}^{2}+X_{0}\right)  ^{k}\right)  ^{\ast
}=\left(  \sum_{j=1}^{m}X_{j}^{2}-X_{0}\right)  ^{k},
\]
where
\[
\sum_{j=1}^{m}X_{j}^{2}-X_{0}%
\]
is still a H\"{o}rmander operator.
\end{remark}

We will prove the following facts.

\begin{theorem}
[Hypoellipticity of $\mathcal{L}$]\label{Thm main hypoellipticity}A
generalized Rockland operator $\mathcal{L}$ is hypoelliptic in $\mathbb{R}%
^{n}$.
\end{theorem}

\begin{theorem}
[Liouville-type theorem for $\mathcal{L}$]\label{Thm main Liouville} Assume
that both $\mathcal{L}$ and $\mathcal{L}^{*}$ are generalized Rockland
operators, and let $\Lambda\in\mathcal{S}^{\prime}(\mathbb{R}^{n})$ be a
tempe\-red distribution satisfying $\mathcal{L}\Lambda=0$ in $\mathcal{S}%
^{\prime}(\mathbb{R}^{n})$. Then, there exists a \emph{polynomial function}
$p=p(x)\in C^{\infty}(\mathbb{R}^{n})$ such that
\begin{equation}
\text{$\Lambda=p$ in $\mathcal{S}^{\prime}(\mathbb{R}^{n})$}.
\label{eq:LambdapolinomioGeneral}%
\end{equation}
In particular, any bounded solution of $\mathcal{L}\Lambda=0$ in
$\mathcal{S}^{\prime}(\mathbb{R}^{n})$ is a constant.
\end{theorem}

\begin{remark}
In Theorem \ref{Thm main Liouville}, as well as in the subsequent Theorems
\ref{thm:ExistenceEstimLhGeneral} and \ref{Thm estimates on Gamma}, we require
that both $\mathcal{L}$ and $\mathcal{L}^{\ast}$ are generalized Rockland
operators. As already noted, in some interesting cases we have $\mathcal{L}%
^{\ast}=\mathcal{L}$; when this is not the case, the two assumptions appear,
as far as we know, to be logically independent: indeed, already on homogeneous
groups, the Rockland property of $\mathcal{L}$ and that of $\mathcal{L}^{\ast
}$ are logically independent in general.
\end{remark}

\begin{remark}
\label{Remark Liouville}A general result by Rothschild \cite{R} shows that a
homogeneous left-invariant operator on a homogeneous group is hypoelliptic if
and only if it satisfies a Liouville property. Later, Luo \cite{L} has proved
that Liouville theorem still holds for homogeneous hypoelliptic operators
(without requiring left invariance). Therefore the above Theorem
\ref{Thm main Liouville} is actually a corollary of Theorem
\ref{Thm main hypoellipticity}, in view of \cite{L}. We think that the short
independent proof of Theorem \ref{Thm main Liouville} that we will present in
Section \ref{subsec Liouville} can have an independent interest.
\end{remark}

\begin{theorem}
[Fundamental solution for $\mathcal{L}$]\label{thm:ExistenceEstimLhGeneral}Let%
\[
\mathcal{L}=\sum_{|I|=\nu}c_{I}X_{I}%
\]
be a generalized Rockland operator, and assume that%
\[
\nu<q,
\]
where $\nu$ is the homogeneity degree of $\mathcal{L}$ and $q$ is the
homogeneous dimension of $\mathbb{R}^{n}$. Assume that also $\mathcal{L}^{*}$
is a generalized Rockland operator.

Then, there exists a \emph{global fundamental solution} $\Gamma(x,y)$ of
$\mathcal{L}$, that is,

\begin{enumerate}
\item[\emph{(a)}] for every fixed $x,y\in\mathbb{R}^{n}$, we have
$\Gamma(\cdot,y)\in L_{\mathrm{loc}}^{1}(\mathbb{R}^{n}),\,\Gamma(x,\cdot)\in
L_{\mathrm{loc}}^{1}(\mathbb{R}^{n})$; also, $\Gamma\in L_{\mathrm{loc}}%
^{1}(\mathbb{R}^{2n})$ (in the joint variables); \vspace{0.1cm}

\item[\emph{(b)}] for every $\varphi\in C_{0}^{\infty}(\mathbb{R}^{n})$, we
have
\[
\int_{\mathbb{R}^{n}}\Gamma(x,y)\mathcal{L}^{\ast}\varphi(x)\,dx=-\varphi
(y)\text{ \ }\forall\,\,y\in\mathbb{R}^{n};
\]

\item[\emph{(c)}] for every $\varphi\in C_{0}^{\infty}(\mathbb{R}^{n})$, the
function%
\[
u\left(  x\right)  =\int_{\mathbb{R}^{n}}\Gamma(x,y)\varphi(y)\,dy
\]
is $L_{loc}^{1}\left(  \mathbb{R}^{n}\right)  $ and satisfies $\mathcal{L}%
u=-\varphi$ in the distributional sense (therefore $u$ is a.e. equal to a
smooth function $u^{\ast}$ satisfying $\mathcal{L}u^{\ast}=-\varphi$ pointwise).
\end{enumerate}

Moreover, $\Gamma$ satisfies the following properties:

\begin{itemize}
\item[\emph{(I)}] setting $\mathbb{D} = \{(x,y)\in\mathbb{R}^{2n}:x = y\}$, we
have $\Gamma\in C^{\infty}(\mathbb{R}^{2n}\setminus\mathbb{D})$.

\item[\emph{(II})] $\Gamma$ is jointly $\delta_{\lambda}$-homogeneous of
degree $\nu-q$, that is,
\[
\Gamma(\delta_{\lambda}(x),\delta_{\lambda}(y))=\lambda^{\nu-q}\,\Gamma
(x,y)\quad\text{for all $(x,y)\in\mathbb{R}^{2n}\setminus\mathbb{D}$ and
$\lambda>0$}.
\]

\item[\emph{(III)}] For every fixed $y\in\mathbb{R}^{n}$, we have
\begin{equation}
\label{eq:Gammavanishesx}\Gamma(x,y)\to0\quad\text{as $|x|\to+\infty$}.
\end{equation}

\item[\emph{(IV)}] The function $\Gamma$ is \emph{unique} within the class of
functions $\gamma\left(  x,y\right)  $ satisfying the properties in points (b)
and (III) above.

\item[\emph{(V)}] For this uniquely defined $\Gamma$, the following identity
holds:
\begin{equation}
\Gamma^{\ast}(x,y)=\Gamma(y,x)\quad\text{for every $x\neq y\in\mathbb{R}^{n}$%
}, \label{eq:GammaGammastarGC}%
\end{equation}
where $\Gamma^{*}$ is the global fundamental solution of $\mathcal{L}^{*}$. In
particular, if $\mathcal{L}$ is \emph{formally self-adjoint} \emph{(}that is,
$\mathcal{L}^{\ast}=\mathcal{L}$), then
\[
\Gamma(x,y)=\Gamma(y,x)\quad\text{for every $x\neq y\in\mathbb{R}^{n}$}.
\]

\end{itemize}
\end{theorem}

\begin{remark}
[About uniqueness of $\Gamma$]The above theorem contains a statement about
uniqueness of $\Gamma$, expressed in terms of the properties of this function:
$\Gamma$ is the only global fundamental solution of $\mathcal{L}$ which
vanishes at infinity. We will see throughout the proof of Theorem 1.13 that
the function $\Gamma$ will be constructed by a three step procedure (lifting
to a homogeneous group - construction of the unique homogeneous global
fundamental solution $\widetilde{\Gamma}$ on that homogeneous group -
in\-te\-gration of $\widetilde{\Gamma}$ by saturation of the variables added
in the lifting procedure) which actually produces exactly one function.

Therefore, our fundamental solution is unique also in a different sense: it is
the fundamental solution which is uniquely produced by our procedure.
\end{remark}

The following last result collects a number of sharp pointwise estimates on
$\Gamma$ and its derivatives.

\begin{theorem}
[Pointwise estimates]\label{Thm estimates on Gamma}Let the assumptions of
Theorem \ref{thm:ExistenceEstimLhGeneral} be in force, and let $r$ be a
nonnegative integer such that
\[
r\geq\nu-n.
\]
Then, the following facts hold.

\begin{itemize}
\item[(1)] (Global upper estimate in the non-critical case). If%
\begin{equation}
r>\nu-n \label{vincolo su r}%
\end{equation}
there exists $c>0$ such that, for every $x,y\in\mathbb{R}^{n}$ \emph{(}with
$x\neq y$\emph{)}, one has
\begin{equation}
\left\vert Z_{1}\cdots Z_{h}\Gamma(x;y)\right\vert \leq c\frac{d_{X}^{\nu
-r}\left(  x,y\right)  }{\left\vert B_{X}\left(  x,d_{X}\left(  x,y\right)
\right)  \right\vert } \label{global upper estimate}%
\end{equation}
for any choice of $Z_{1},\ldots,Z_{h}$ (with $h\leq r$) in%
\[
\left\{  X_{1}^{x},X_{2}^{x},...,X_{m}^{x},X_{1}^{y},X_{2}^{y},...,X_{m}%
^{y}\right\}
\]
satisfying
\[
\textstyle\sum_{i=1}^{h}\left\vert Z_{i}\right\vert =r
\]
(where $\left\vert Z_{i}\right\vert =\nu_{k_{i}}$ if $Z_{i}=X_{k_{i}}$ for
some $1\leq k_{i}\leq m$).

In particular, for every fixed $x\in\mathbb{R}^{n}$ one has%
\[
\lim_{\left\vert y\right\vert \rightarrow+\infty}Z_{1}\cdots Z_{h}%
\Gamma(x;y)=0.
\]

\item[(2)] (Local upper estimate in the critical case). If
\[
r=\nu-n
\]
then for every compact set $K\subset\mathbb{R}^{n}$ there exist positive
constants $C_{0}\ $and $R_{0}=R_{0}(K)$ such that for every $x,y\in
K\mathbb{\ }$\emph{(}with $x\neq y$\emph{)}, one has (with the same meaning of
$Z_{1},...,Z_{h}$)%
\begin{equation}
\left\vert Z_{1}\cdots Z_{h}\Gamma(x;y)\right\vert \leq C_{0}\frac{d_{X}%
^{n}\left(  x,y\right)  }{\left\vert B_{X}\left(  x,d_{X}\left(  x,y\right)
\right)  \right\vert }\log\left(  \frac{R_{0}}{d_{X}\left(  x,y\right)
}\right)  . \label{upper logarithmic}%
\end{equation}

\end{itemize}

As usual, in both \eqref{global upper estimate}-\eqref{upper logarithmic} we
understand that
\[
\text{$Z_{1}\cdots Z_{h}\Gamma(x,y)=\Gamma(x,y)$ when $r=0$}.
\]

\end{theorem}

\begin{remark}
[Assumptions on $\mathcal{L}$]We stress the fact that all our Theorems
\ref{Thm main hypoellipticity}, \ref{Thm main Liouville},
\ref{thm:ExistenceEstimLhGeneral}, \ref{Thm estimates on Gamma} will be proved
assuming that $\mathcal{L}$ is a generalized Rockland operator in the sense of
Definition \ref{Def gen Rockland 2} which, as already discussed, is more
general than Definition \ref{Def gen Rockland 1}.
\end{remark}

\begin{remark}
[Assumptions on $\nu$]The validity of assumption (\ref{vincolo su r}) is
automatic if $\nu<n$ (a condition which is stronger than $\nu<q$, that we are
assuming in Theorem \ref{thm:ExistenceEstimLhGeneral}). If instead%
\[
n\leq\nu<q,
\]
then a jointly homogeneous fundamental solution $\Gamma\left(  x;y\right)  $
still exists, by Theorem \ref{thm:ExistenceEstimLhGeneral}, but the upper
bounds (\ref{global upper estimate}) are assured only for derivatives of order
$r$ \emph{large enough} (see Example \ref{Ex large enoguh}). We stress the
fact that, differently from Theorems \ref{thm:ExistenceEstimLhGeneral} and
\ref{Thm estimates on Gamma}, our Theorems \ref{Thm main hypoellipticity} and
\ref{Thm main Liouville} do not assume the relation $\nu<q$.

Instead, under the stronger assumption $\nu<n$, point (c)\ in Theorem
\ref{thm:ExistenceEstimLhGeneral} can be improved. Actually, by
(\ref{global upper estimate}) we know that
\[
\left\vert \Gamma(x;y)\right\vert \leq c\frac{d_{X}^{\nu}\left(  x,y\right)
}{\left\vert B_{X}\left(  x,d_{X}\left(  x,y\right)  \right)  \right\vert }.
\]
This bound allows to prove that for every $\varphi\in C_{0}^{\infty
}(\mathbb{R}^{n})$, the function%
\[
u\left(  x\right)  =\int_{\mathbb{R}^{n}}\Gamma(x,y)\varphi(y)\,dy
\]
is actually continuous, and then smooth (and not only a.e.\,equal to a smooth function).
\end{remark}

\begin{remark}
[About the vanishing at infinity of $\Gamma$]The fact that $\Gamma\left(
x,\cdot\right) $ vanishes at infinity is proved both in Theorem
\ref{thm:ExistenceEstimLhGeneral}\,-\,\emph{(III)} and also, in a more
quantitative way, in Theorem \ref{Thm estimates on Gamma} (point (1) with
$r=0$). However, note that the second, stronger, conclusion (the one in
Theorem \ref{Thm estimates on Gamma}) requires the assumption $\nu\leq n$,
which is generally stronger than the assumption $v<q$ required in Theorem
\ref{thm:ExistenceEstimLhGeneral}.
\end{remark}

The strategy we will use to prove our results consists in generalizing the
lifting technique which has been firstly devised by Biagi and Bonfiglioli in
\cite{BB-lift} for homogeneous sublaplacians, implementing it in the present
more general situation to build a homogeneous fundamental solution. We think
that this technique has an independent interest and can be fruitful also in
the future.

Our main results can be extended also to heat-type operators of the kind
$\mathcal{L}\pm\partial_{t}$, provided $\mathcal{L}$ is a generalized Rockland
operator satisfying a further positivity assumption. This will be performed in
Section \ref{sec heat type operators}, and we refer to that section for the
exact statements. \bigskip

\noindent\textbf{Examples and comparison with the existing literature}
\vspace{0.1cm}

\noindent As already said, higher-order hypoelliptic differential operators
with variable coefficients have been mainly studied in the context of
homogeneous groups. Rockland in \cite{Ro} proves a sufficient condition for
the hypoellipticity of a left-invariant homogeneous operator on the Heisenberg
groups $\mathbb{H}^{n}$ and applies his result to show, as an interesting
example, the hypoellipticity of the operator%
\[
L=\sum_{i=1}^{n}\left(  X_{i}^{2m}+Y_{i}^{2m}\right)
\]
for any positive integer $m$\ (where $X_{i},Y_{i}$ are the standard generators
of $\mathbb{H}^{n}$). Helffer-Nourrigat \cite{HN} prove, as an application of
their result that we have already discussed, the hypoellipticity of the
operators (\ref{Rock}). Another class of special examples of higher-order,
homogeneous left-invariant operators on homogeneous groups, which have been
studied in some detail, is that of (positive integer) \emph{powers of
sublaplacians on Carnot groups}. For these operators hypoellipticity is
obvious, while an interesting issue is the construction and study of a global
fundamental solution. In this context, Benson, Dooley and Ratcliff in
\cite{BDR} have computed the fundamental solution of the $k$-th power of the
sublaplacian on the Heisenberg group $\mathbb{H}^{n}$. Later, Kumar and Mishra
in \cite{KM} have computed the fundamental solution of the $k$-th power of the
sublaplacian on any step-2 nilpotent Lie groups.

As for homogeneous but not left-invariant higher-order operators, we can only
mention the paper by Grushin \cite{Gr}, where it is proved the hypoellipticity
of a very special class of operators, whose main prototype is the following:%
\begin{equation}
\mathcal{L}=\Delta_{x_{1}}^{k}+\left\vert x_{1}\right\vert ^{2h}\Delta_{x_{2}%
}^{k} \label{Grushin}%
\end{equation}
with $h,k$ positive integers and $x=\left(  x_{1},x_{2}\right)  \in
\mathbb{R}^{n_{1}+n_{2}}$. These operators fit our assumptions when $h/k$ is integer.

To the best of our knowledge, the present paper is the first study of a
general class of \emph{higher-order hypoelliptic operators }which are
structured on H\"{o}rmander vector fields but \emph{not left-invariant on any
homogeneous group}.

We can show, through examples, that our results in Theorems
\ref{thm:ExistenceEstimLhGeneral} and \ref{Thm estimates on Gamma} are
consistent with the known results for left-invariant homogeneous higher-order
hypoelliptic operators on homogeneous groups \emph{or} for second order
homogeneous hypoelliptic operators (not left-invariant w.r.t.\,any Lie group
structure). On the other hand, our results cover many situations which have
not been studied so far.

\medskip

The limitation $\nu<q$ in Theorem \ref{thm:ExistenceEstimLhGeneral}, for the
\emph{existence} of a global fundamental solution jointly homogeneous in
$\left(  x,y\right)  $, is consistent with the results which are known for
left-invariant homogeneous operators on homogeneous groups. Actually, Folland
\cite{Fo2} has proved the existence of a homogeneous fundamental solution
under the assumption $\nu<q$, while Geller in \cite[Thm.\,3]{G} has proved
that, when $\nu\geq q$, a global fundamental solution still exists, but is not
homogeneous; actually, $\Gamma$ in that situation has the following structure%
\[
\Gamma\left(  x\right)  =K\left(  x\right)  +p\left(  x\right)  \log\left\Vert
x\right\Vert
\]
where $K\left(  x\right)  $ is a homogeneous distribution of degree $\nu-q$,
$p\left(  x\right)  $ is a polynomial of degree $\nu-q$ and $\left\Vert
\cdot\right\Vert $ is a homogeneous norm.

\begin{example}
About the necessity of the condition $\nu<q$ for the existence of a
homogeneous fundamental solution.

a) The Laplace operator in $\mathbb{R}^{2}$, where%
\[
\nu=q=n=2
\]
and%
\[
\Gamma\left(  x,y\right)  =c\log\left\vert x-y\right\vert
\]
is not homogeneous.

b) The $k$-th power of the Laplacian in $\mathbb{R}^{n}$, $\Delta^{k}$, where:%
\[
\nu=2k,n=q
\]
and when $2k\geq n$ (so that $\nu\geq q$) and $n$ is even, the fundamental
solution is%
\[
\Gamma\left(  x,y\right)  =\Gamma_{0}\left(  x-y\right)  \text{ with }%
\Gamma_{0}\left(  x\right)  =c_{k,n}\left\vert x\right\vert ^{2k-n}
\log\left\vert x\right\vert
\]
(not homogeneous). For this explicit formula, see e.g. \cite[p.\,48]{GGS}.
Note that these examples a) and b) are consistent with the aforementioned
Geller's result.

c) The $k$-th power of the sublaplacian on the Heisenberg group $\mathbb{H}%
^{n},$ $\Delta_{\mathbb{H}}^{k}$. For this operator Benson, Dooley and
Ratcliff in \cite{BDR}, as already recalled, have computed the explicit form
of a homogeneous fundamental solution under the restriction $k\leq n$. Here:%
\[
q=2n+2;\nu=2k,
\]
therefore%
\[
\nu<q\Leftrightarrow k\leq n\text{.}%
\]

\end{example}

\medskip

The further limitation $r>\nu-n$ appearing in Theorem
\ref{Thm estimates on Gamma}, \emph{for the validity of pointwise upper
estimates on} $\Gamma$\emph{ and its derivatives}, is consistent with the
results proved by the Authors and Bonfiglioli in \cite{BBB-fundsol} for second
order homogeneous (and not left-invariant) hypoelliptic operators. Let us
recall here an example studied in \cite[Exm.\,6.7]{BBB-fundsol}:

\begin{example}
We consider the vector fields on $\mathbb{R}^{2}$
\[
X_{1}=\partial_{x_{1}},\quad X_{2}=x_{1}\,\partial_{x_{2}},
\]
which are homogeneous of degree $1$ with respect to the dilations
\[
\delta_{\lambda}(x_{1},x_{2})=(\lambda x_{1},\lambda^{2}x_{2}),
\]
and the second order operator
\[
L=\partial_{x_{1}}^{2}+x_{1}^{2}\,\partial_{x_{2}}^{2}.
\]
Here
\[
\nu=n=2,\text{ }q=3\text{, so }\nu=n<q\text{,}%
\]
hence (by the results in \cite{BB-lift}, \cite[Exm.\,6.7]{BBB-fundsol}, or by
Theorem \ref{thm:ExistenceEstimLhGeneral}) there exists a global fundamental
solution $\Gamma\left(  x,y\right)  $, jointly homogeneous of degree
$\nu-q=-1$.

However, in this situation the global upper bound (\ref{global upper estimate}%
) only holds for $r\geq1$, that is when we are actually estimating the
\emph{derivatives} of $\Gamma$ (see point (1) in Theorem
\ref{Thm estimates on Gamma}). Instead, the function $\Gamma$ itself only
satisfies local estimates of power / logarithmic type (see
(\ref{upper logarithmic}) in point (2) in Theorem \ref{Thm estimates on Gamma}%
). For this operator the explicit form of the fundamental solution is actually
known, and satisfies sharply the bound (\ref{upper logarithmic}). We refer to
\cite[Exm.\,6.7]{BBB-fundsol} for details.
\end{example}

Let us end with a couple of examples which fit the assumptions of the present
paper, and not those of the previous theories.

\begin{example}
\label{Ex large enoguh}In $\mathbb{R}^{n}$, let us consider the system of
H\"{o}rmander vector fields
\begin{align*}
X_{1}  &  =\partial_{x_{1}}\\
X_{2}  &  =x_{1}\partial_{x_{2}}+x_{2}\partial_{x_{3}}+\ldots+x_{n-1}%
\partial_{x_{n}},
\end{align*}
which are $1$-homogeneous w.r.t.\,the dilations%
\[
\delta_{\lambda}(x)=(\lambda x_{1},\lambda^{2}x_{2},\cdots,\lambda^{n}x_{n}),
\]
and let us consider the fourth order operator%
\[
L=X_{1}^{4}+X_{2}^{4}.
\]
Here%
\[
q=\frac{n\left(  n+1\right)  }{2},\nu=4
\]
so for every dimension $n\geq5$ the conditions
\[
\nu<n<q
\]
hold. To fix ideas, pick $n=5$, so $q=15$. Then there exists a global
fundamental solutions jointly homogeneous of degree
\[
\nu-q=-11.
\]
Moreover, since $\nu-n=-1<0$, the global upper bound
(\ref{global upper estimate}) in the non-critical case holds for every
$r\geq0.$
\end{example}

\begin{example}
We consider, for a fixed positive integer $k$, the system of H\"{o}r\-man\-der
vector fields in $\mathbb{R}^{2}$
\[
X_{1}=\partial_{x_{1}},\quad X_{2}=x_{1}^{k}\,\partial_{x_{2}},
\]
which are homogeneous of degree $1$ with respect to the dilations
\[
\delta_{\lambda}(x_{1},x_{2})=(\lambda x_{1},\lambda^{k+1}x_{2}),
\]
and the fourth order operator
\[
L=X_{1}^{4}+X_{2}^{4}=\partial_{x_{1}}^{4}+x_{1}^{4k}\,\partial_{x_{2}}^{4}.
\]
Here
\[
n=2;\nu=4;q=k+2,
\]
so
\[
\nu<q\Leftrightarrow k>2\text{.}%
\]

For $k=1,2,$ we cannot expect the existence of a jointly homogeneous global
fundamental solution, while, for every integer $k\geq3$, by Theorem
\ref{thm:ExistenceEstimLhGeneral} there exists a global fundamental solution
$\Gamma\left(  x,y\right) $, jointly homogeneous of degree $\nu-q=2-k$.

Under this assumption, the global upper bound (\ref{global upper estimate}) in
the non-critical case holds for
\[
r>\nu-n=2,
\]
that is when we are actually estimating the \emph{derivatives} of $\Gamma$ of
order at least $3$ (see point (1) in Theorem \ref{Thm estimates on Gamma}).

The critical case, with local estimates of power / logarithmic type (see
(\ref{upper logarithmic}) in point (2) in Theorem \ref{Thm estimates on Gamma}%
) corresponds to $r=2$, that is upper bounds on second order derivatives.

The function $\Gamma$ itself, and its first order derivatives, despite their
being jointly homogeneous, do not satisfy an easy upper bound in terms of
$d_{X}\left(  x,y\right)  $.
\end{example}

\bigskip

\noindent\textbf{Plan of the paper.} In Section \ref{sec preliminaries} we
collect some known facts on a couple of topics: first, in Section
\ref{sec homogeneous vector fields}, we recall some facts about the control
distance and the volume of balls in the context of homogeneous (but not left
invariant) H\"{o}rmander vector fields; then, in Section
\ref{sec left-invariant}, we deal with homogeneous groups, homogeneous and
left-invariant hypoelliptic operators in this context, and their fundamental
solution. In Section \ref{sec:PropGenRockland} we prove our results. First, in
Section \ref{sec lifting}, we present the lifting technique devised in
\cite{BB-lift} adapted to the present context. Applying this technique, in
Sections \ref{sec hypoellipticity} and \ref{subsec Liouville} we prove
Theorems \ref{Thm main hypoellipticity} and \ref{Thm main Liouville},
respectively. In Section \ref{sec existence fund sol} we prove Theorems
\ref{thm:ExistenceEstimLhGeneral} and \ref{Thm estimates on Gamma}, that is
our main results. Finally, in Section \ref{sec heat type operators} we show
how our main results can be extended to heat-type operators of the kind
$\mathcal{L}\pm\partial_{t}$. \bigskip

\noindent-\,\,\textbf{Conflict of interest.} The author state no conflict of interest.

\noindent-\,\,\textbf{Data availability statement.} There are no data
associated with this research.

\noindent

\section{Preliminaries and known results\label{sec preliminaries}}

\subsection{Homogeneous H\"{o}rmander vector fields and their weighted control
distance}

\label{sec homogeneous vector fields}

From now on we consider a family
\[
X=\left\{  X_{1},...,X_{m}\right\}
\]
of real smooth vector fields in $\mathbb{R}^{n}$, satisfying Assumption
\textbf{(H1)} in Section \ref{sec:Intro}, and we keep the notation introduced there.

We have already noted in section \ref{sec:Intro} that our Assumption
\textbf{(H1)} imply that H\"{o}rmander's condition actually holds in the whole
$\mathbb{R}^{n}$. In the following Remark we collect some other easy
consequences of \textbf{(H1)}.

\begin{remark}
\label{rem:ConseguenzeH1H2}The following assertions hold.

\begin{enumerate}
\item Since the vector fields $X_{1},\ldots,X_{m}$ are smooth and
$\delta_{\lambda}$-homogeneous of \emph{po\-si\-tive degree }$\nu_{1}%
,\ldots,\nu_{m}$, it is easy to see that
\begin{equation}
X_{j}=\sum_{i=1}^{n}p_{i,j}(x)\partial_{x_{i}}\qquad(1\leq j\leq m),
\label{eq:expressionXj}%
\end{equation}
where $p_{i,j}$ are \emph{polynomial functions}, $\delta_{\lambda}%
$-homogeneous of degree
\[
\sigma_{i}-\nu_{j}\leq\sigma_{i}-1.
\]
In particular, since $\sigma_{1},\ldots,\sigma_{n}$ are increasingly ordered,
we deduce that
\begin{equation}
\text{\emph{$p_{i,j}$ is independent of $x_{i},x_{i+1},\ldots,x_{n}$}.}
\label{eq:piramydshape}%
\end{equation}
This fact also implies that%
\begin{equation}
X_{i}^{\ast}=-X_{i}\quad\text{for every $i=1,\ldots,m.$} \label{eq:XistarXi}%
\end{equation}

\item An analogous homogeneity reasoning shows that no vector field can be
homogeneous of a degree larger than $\sigma_{n}$ (without being identically
zero). Since, on the other hand, the commutator of two $\delta_{\lambda}%
$-homo\-geneous vector fields of degrees $d_{1},d_{2}$ is itself
$\delta_{\lambda}$-homogeneous of degree $d_{1}+d_{2}$, we see that the Lie
algebra $\operatorname*{Lie}(X)$ generated by $X$ is \emph{nilpotent of step
$r\leq\sigma_{n}$}.

\item Since $\mathrm{Lie}(X)$ is finitely generated and nilpotent, we deduce
that
\[
N=\mathrm{dim}(\mathrm{Lie}(X))<+\infty.
\]
On the other hand, by Assumption \textbf{(H1)} we also have $N\geq n$.
\end{enumerate}
\end{remark}

Next, we specialize to our context the notion of \emph{weighted control
distance} induced by a family of vector fields, first studied in some
generality by Nagel-Stein-Wainger in \cite{NSW}.

\begin{definition}
\label{Def CC distance}Given our system $X=\left\{  X_{1},...,X_{m}\right\}  $
of vector fields, for every fixed $\delta>0$ and every $x,y\in\mathbb{R}^{n}$,
we denote by $C_{x,y}(\delta)$ the set of the \emph{absolutely continuous
curves}
\[
\gamma:[0,1]\rightarrow\mathbb{R}^{n}%
\]
satisfying the following properties:

\begin{enumerate}
\item[i)] $\gamma(0)=x$ and $\gamma(1)=y$;

\item[ii)] for a.e.\thinspace$t\in\lbrack0,1]$ one has
\[
\gamma^{\prime}(t)=\sum_{i=1}^{m}a_{i}(t)X_{i}(\gamma(t)),\text{ with
}\left\vert a_{i}(t)\right\vert \leq\delta^{\nu_{i}}%
\]
$\quad$for every $1\leq i\leq$$m.$
\end{enumerate}

Then, we define the \emph{weighted $X$-control distance} as follows:
\begin{equation}
d_{X}(x,y)=\inf\left\{  \delta>0:\,\exists\,\,\gamma\in C_{x,y}(\delta
)\right\}  , \label{CC distance}%
\end{equation}
where the nonemptiness of $C_{x,y}(\delta)$ follows from the connectivity
property of a system of H\"{o}rmander vector fields (Chow's theorem, for a
proof see e.g. \cite[Sec.\,1.6]{BBbook}).
\end{definition}

It can be proved that $d_{X}$ is actually a distance in $\mathbb{R}^{n}$. (For
the standard cases (a): $\nu_{i}=1$ for every $i$, and (b): $\nu_{1}=2$ and
$\nu_{i}=1$ for $i=2,3,...,m$ a proof can be found in \cite[Sec.\,1.4,
1.5]{BBbook}. The general case is similar). Moreover, this $d_{X}$ is
topologically, but not metrically, equivalent to the Euclidean distance.

For every fixed $a\in\mathbb{R}^{n}$ and every $r>0$, we denote by
$B_{X}(a,r)$ the $d_{X}$-ball with centre $a$ and radius $r$:
\[
B_{X}(a,r)=\{x\in\mathbb{R}^{n}:\,d_{X}(x,a)<r\}.
\]
The next proposition collects some important and nontrivial properties of
these metric balls, which improve in our situation a deep result proved for a
general system of H\"ormander vector fields, but in local form, by
Nagel-Stein-Wainger \cite{NSW}.

\begin{proposition}
[Geometry of $d_{X}$-balls]\label{rem:dXBalls}(See \cite[Thm.\,B]%
{BBB-fundsol}). Let $X$ and $d_{X}$ be as above. Then there exist constants
$\gamma_{1},\gamma_{2}>0$ such that
\begin{equation}
\gamma_{1}\sum_{k=n}^{q}f_{k}(x)r^{k}\leq\left\vert B_{X}\left(  x,r\right)
\right\vert \leq\gamma_{2}\sum_{k=n}^{q}f_{k}(x)r^{k}.
\label{eq:measureBallsNSW}%
\end{equation}
for every $x\in\mathbb{R}^{n}$ and $r>0$. \vspace{0.1cm}

Here, $q$ is as in \eqref{q}. Furthermore, for every $k=n,\ldots,q$,
$f_{k}:\mathbb{R}^{n}\rightarrow\mathbb{R}$ is a function which is continuous,
non-ne\-ga\-tive and $\delta_{\lambda}$-ho\-mogeneous of degree $q-k$. In
particular, $f_{q}(\cdot) $ is a positive \emph{constant}.
\end{proposition}

The previous Proposition in particular implies the validity of a \emph{global
doubling condition}: there exists a constant $c>0$ such that%
\begin{equation}
\left\vert B_{X}\left(  x,2r\right)  \right\vert \leq c\left\vert B_{X}\left(
x,r\right)  \right\vert \label{doubling}%
\end{equation}
for every $x\in\mathbb{R}^{n}$ and $r>0$. In turn, property (\ref{doubling}),
by a standard computation which holds in doubling spaces, in particular
implies the following:

\begin{proposition}
\label{Prop fractional integral}For every $\alpha>0$ there exists a constant
$c>0$ such that for every $x\in\mathbb{R}^{n}$ and $r>0$ one has:%
\begin{equation}
\int_{d_{X}\left(  x,y\right)  <r}\frac{d_{X}\left(  x,y\right)  ^{\alpha}%
}{\left\vert B_{X}\left(  x,d_{X}\left(  x,y\right)  \right)  \right\vert
}dy\leq cr^{\alpha}\text{.} \label{integrale frazionario}%
\end{equation}

\end{proposition}

\subsection{Homogeneous groups and left-invariant
operators\label{sec left-invariant}}

We start by recalling the following

\begin{definition}
[Homogeneous group]\label{Def homogeneous group}Assume we have, in
$\mathbb{R}^{n}$, a Lie group structure $\left(  \mathbb{R}^{n},\star\right)
$ and a family of (diagonal, nonisotropic) dilations $\left\{  D_{\lambda
}\right\}  _{\lambda>0}$,
\begin{equation}
D_{\lambda}\left(  x\right)  =\left(  \lambda^{\sigma_{1}}x_{1},\lambda
^{\sigma_{2}}x_{2},...,\lambda^{\sigma_{n}}x_{n}\right)  , \label{dilations}%
\end{equation}
for some exponents $\sigma_{i}$ with%
\[
0<\sigma_{1}\leq\sigma_{2}\leq...\leq\sigma_{n},
\]
which are group automorphisms; we then say that
\[
\mathbb{G}=\left(  \mathbb{R}^{n},\star,\left\{  D_{\lambda}\right\}
_{\lambda>0}\right)
\]
is a \emph{homogeneous group}, of \emph{homogeneous dimension}%
\[
Q=\sum_{i=1}^{n}\sigma_{i}.
\]
Up to normalization of the exponents, it is not restrictive to assume
$\sigma_{1}=1$; moreover, under the assumptions that we will make later,
$\sigma_{i}$ will be positive \emph{integers}.
\end{definition}

For every $x\in\mathbb{G}$, let$\ \tau_{x}$ denote the left-translation by
$x$, on $\mathbb{G}$, that is%
\begin{align*}
\tau_{x}  &  :\mathbb{R}^{n}\rightarrow\mathbb{R}^{n},\\
\tau_{x}(y)  &  :=x\star y.
\end{align*}
Then, we say that a smooth differential operator $Z$ on $\mathbb{G}$ is
\emph{left-invariant }if%
\begin{equation}
Z\left(  u\circ\tau_{x}\right)  =(Zu)\circ\tau_{x}, \label{eq.defXleftinv}%
\end{equation}
for every $u\in C^{\infty}(\mathbb{R}^{n})$ and $x\in\mathbb{R}^{n}$. The
notion of homogeneity of some degree $\nu$ (w.r.t.\,the dilations $D_{\lambda
}$) for a differential operator on $\mathbb{G}$ is analogous to
(\ref{homogeneous op}).

Here we collect a few notions and known properties about homogeneous groups.
For more details and proofs, see for instance \cite[Chap.\,3]{BBbook}.

On any homogeneous group,\ it is not restrictive to assume that the neutral
element of $\mathbb{G}$ is $0$. The $n$-dimensional Lebesgue measure is a
\emph{bi-invariant Haar measure} on $\mathbb{G}$ (see \cite[Thm.\,3.8]{BBbook}).

\begin{definition}
\label{Def homogeneous norm}A \emph{homogeneous norm} on $\mathbb{G}$ is a
continuous function
\[
\left\Vert \cdot\right\Vert :\mathbb{G\rightarrow\lbrack}0,+\infty)
\]
such that:

\begin{enumerate}
\item $\left\Vert x\right\Vert =0\,\Leftrightarrow\,x=0$;

\item $\left\Vert D_{\lambda}\left(  x\right)  \right\Vert =\lambda\left\Vert
x\right\Vert $ for every $x\in\mathbb{R}^{n},\,\lambda>0$.
\end{enumerate}
\end{definition}

It can be proved that every homogeneous norm also satisfies%
\begin{equation}%
\begin{tabular}
[c]{l}%
$\left\Vert x^{-1}\right\Vert \leq c\left\Vert x\right\Vert $\\
$\left\Vert x\star y\right\Vert \leq c\left(  \left\Vert x\right\Vert
+\left\Vert y\right\Vert \right)  $%
\end{tabular}
\label{norm properties}%
\end{equation}
for some absolute constant $c>0$ and every $x,y\in\mathbb{G}$.

On every homogeneous group there exist infinitely many different homogeneous
norms, all of them being equivalent. For instance, an explicit example of
homogeneous norm is the following:%
\begin{equation}
\left\Vert x\right\Vert =\sum_{i=1}^{n}\left\vert x_{i}\right\vert
^{1/\sigma_{i}}. \label{canonical homogeneous norm}%
\end{equation}

If $\left\Vert \cdot\right\Vert $ is a homogeneous norm such that $\left\Vert
x\right\Vert =\left\Vert x^{-1}\right\Vert $ for every $x\in\mathbb{R}^{n}$
(where $x^{-1}$ denotes the inverse of $x$ with respect to $\star$), we say
that $\left\Vert \cdot\right\Vert $ is \emph{symmetric}. Any symmetric
homogeneous norm induces a \emph{quasidistance}%
\begin{equation}
d\left(  x,y\right)  =\left\Vert x^{-1}\star y\right\Vert ,
\label{quasidistance}%
\end{equation}
satisfying

\begin{enumerate}
\item $d(x,y)\geq0$ for every $x,y\in\mathbb{R}^{n}$, and
$d(x,y)=0\,\Leftrightarrow\,x=y$;

\item $d(x,y)=d(y,x)$ for every $x,y\in\mathbb{R}^{n}$;

\item there exists a constant $c\geq1$ such that
\[
d(x,z)\leq c\left(  d(x,y)+d(y,z)\right)  \ \text{for every $x,y,z\in
\mathbb{R}^{n}.$}%
\]

\end{enumerate}

\begin{remark}
[Homogeneous norms with or without translations]The notion of homogeneous
norm, actually, can be introduced in $\mathbb{R}^{n}$ whenever a family of
dilations $\left\{  \delta_{\lambda}\right\}  $ is defined, as in
(\ref{delta lambda}), even when a group structure is lacking; the norm
(\ref{canonical homogeneous norm}) is still a good example also in this more
general context. In that situation, clearly, the properties
(\ref{norm properties}) and the quasidistance (\ref{quasidistance}) are no
longer meaningful. We will use homogeneous norms in this more general context
in Section \ref{sec existence fund sol}.
\end{remark}

We denote by $\mathrm{Lie}(\mathbb{G})$ the set of the left\--in\-va\-riant
vector fields on $\mathbb{G}$ and we call it the \emph{Lie algebra of
$\mathbb{G}$}. This Lie algebra has finite dimension (as a real vector space),
and
\[
\dim(\mathrm{Lie}(\mathbb{G}))=n.
\]

Assume that $X_{1},...,X_{m}$ is a set of \emph{generators} of $\mathrm{Lie}%
(\mathbb{G})$, that is a set of left-invariant vector fields satisfying
H\"{o}rmander's condition. Moreover, assume that each $X_{i}$ is $D_{\lambda}%
$-homogeneous of positive integer degree $\nu_{i}$ with $1\leq\nu_{1}\leq
\nu_{2}\leq...\leq\nu_{m}$. In other words, we now assume that $X_{1}%
,...,X_{m}$ is a system of vector fields satisfying our Assumption
\textbf{(H1)}, and moreover they are left-invariant.

In the special case $\nu_{i}=1$ for $i=1,2,...,m$ we say that $\mathbb{G}$ is
a \emph{stratified group}, or a \emph{Carnot group}.

In the other special case $\nu_{i}=1$ (for $i=1,...,m-1$) and $\nu_{m}=2$, we
say that $\mathbb{G}$ is \emph{stratified of type II.}

In this paper we are mainly interested in the general situation when these
cases \emph{do not} occur.

\begin{notation}
Note that in the following we will always use $D_{\lambda}$ to denote
dilations in a homogeneous group, and $\delta_{\lambda}$ to denote a family of
dilations in $\mathbb{R}^{n}$ when we do not assume the existence of a group
structure. Also, we will use $Q$ to denote the homogeneous dimension of a
homogeneous group, and $q$ to denote the homogeneous dimension related to a
family of dilations in $\mathbb{R}^{n}$ when a group structure is lacking. The
reason of this notation will be clear in Section \ref{sec lifting}.
\end{notation}

In this context, the weighted control distance $d_{X}$ (see Definition
\ref{Def CC distance}) is metrically equivalent to the quasidistance $d$
induced by any homogeneous norm in $\mathbb{G}$ (see (\ref{quasidistance})).
Actually, in this case%
\[
\left\Vert x\right\Vert =d_{X}\left(  x,0\right)
\]
is a homogeneous norm. This fact and other properties of $d_{X}$ on
homogeneous groups are proved for instance in \cite[Sec.\,3.5]{BBbook} (the
proofs are given in the case $\nu_{i}=1$ for every $i$, but the same arguments
work in our situation).

\begin{remark}
\label{rem:HomNormLloc}We explicitly mention, for a future reference, the
following (local) \emph{integrability property} of a general homogene\-ous
norm $\left\Vert \cdot\right\Vert $ on a given homogeneous group
$\mathbb{G}=(\mathbb{R}^{n},\star,D_{\lambda})$ (see, e.g., \cite[Cor.\,5.4.5]%
{BLUBook}). \vspace{0.1cm}

\emph{Given any $\alpha\in\mathbb{R}$, we have
\[
\left\Vert \cdot\right\Vert ^{\alpha}\in L_{\mathrm{loc}}^{1}(\mathbb{R}%
^{n})\,\,\Longleftrightarrow\,\,\alpha>-Q
\]
where $Q>0$ is the homogeneous dimension of $\mathbb{G}$.}
\end{remark}

\bigskip

Starting with the generators $X_{1},...,X_{m}$ of $\mathrm{Lie}(\mathbb{G})$,
where we are assuming each $X_{i}$ $D_{\lambda}$-homogeneous of degree
$\nu_{i}$, let us consider a differential operator of the kind%
\begin{equation}
L=\sum_{\left\vert I\right\vert =\nu}c_{I}X_{I} \label{general L}%
\end{equation}
where $c_{I}$ are real constants and $\nu$ is a fixed positive integer. By
construction, $L$ is left-invariant and $\nu$-homogeneous. If $L$ is also
hypoelliptic, then $L$ is a Rockland operator (see Definition
\ref{Def Rockland}). In the following it will be important to discuss also the
hypoellipticity of the transpose $L^{\ast}$. Note that%
\[
L^{\ast}=\sum_{\left\vert I\right\vert =\nu}c_{I}\left(  X_{I}\right)  ^{\ast}%
\]
and, if
\[
X_{I}=X_{i_{1}}X_{i_{2}}...X_{i_{k}}%
\]
then%
\[
\left(  X_{I}\right)  ^{\ast}=X_{i_{k}}^{\ast}...X_{i_{2}}^{\ast}X_{i_{1}%
}^{\ast}=\left(  -1\right)  ^{k}X_{i_{k}}...X_{i_{2}}X_{i_{1}}.
\]

For instance, as already noted in Remark \ref{remark trasposti notevoli}, for
Rockland operators of type (\ref{power of Rockland}) one simply has $L^{\ast
}=L$, while for (a powers of) a H\"{o}rmander operators%
\[
L=\left(  \sum_{j=2}^{m}X_{j}^{2}+X_{1}\right)  ^{k}%
\]
one has%
\[
L^{\ast}=\left(  \sum_{j=2}^{m}X_{j}^{2}-X_{1}\right)  ^{k},
\]
which is still a (power of) a H\"{o}rmander operator, in particular it is
still hypoelliptic.

\bigskip

A property which will be useful is the following (see \cite[Thm.\,3.2.45]{FR}).

\begin{theorem}
[Liouville-type theorem on homogeneous groups]\label{thm:LiouvilleGroups} Let
$L$ be a homogeneous (of any degree) left-invariant differential operator on a
homogeneous group $\mathbb{G}$. We assume that $L$ and $L^{\ast}$ are
hypoelliptic on $\mathbb{G}$. If the tempered distribution $f\in
\mathcal{S}^{\prime}(\mathbb{G})$ satisfies $Lf=0$ then $f$ is a polynomial.
\end{theorem}

\medskip

Next, for a hypoelliptic operator we are interested in results assuring the
existence of a global fundamental solution with good properties. The following
result holds:

\begin{theorem}
[Fundamental solution on homogeneous groups]\label{thm:PropertiesGammGroups}%
Let $L$ be a $\nu$-ho\-mo\-geneous left-invariant differential operator on a
homogeneous group
\[
\mathbb{G}=(\mathbb{R}^{n},\star,\{D_{\lambda}\}),
\]
with homogeneous dimension $Q$. We assume that both $L$ and $L^{\ast}$ are
hypoelliptic on $\mathbb{G}$, and that $0<\nu<Q. $ Then, there exists a
\emph{unique $\Gamma_{0}\in\mathcal{S}^{\prime}(\mathbb{G})$} (tempered
distribution), \emph{$D_{\lambda}$-homogeneous of degree $\nu-Q$}, satisfying
\[
L\Gamma_{0}=-\delta_{0}\quad\text{in }\mathcal{S}\text{$^{\prime}(\mathbb{G}%
)$},
\]
where $\delta_{0}$ is the Dirac distribution centered at $0$. This $\Gamma
_{0}$ is called \emph{global fun\-da\-mental solution of $L$}.

Furthermore, the following assertions hold.

\begin{enumerate}
\item[(1)] (Smoothness). $\Gamma_{0}$ is smooth on $\mathbb{G}\backslash\{0\}$.

\item[(2)] (Global pointwise estimate). There exists $c>0$ such that
\begin{equation}
\left\vert \Gamma_{0}(x)\right\vert \leq c\Vert x\Vert^{\nu-Q}\text{ \ \ for
every $x\in\mathbb{R}^{n},\,x\neq0$}. \label{eq:GammazeroEstimate}%
\end{equation}
In particular, we have
\begin{equation}
\Gamma_{0}(x)\rightarrow0\text{ \ \ as $\Vert x\Vert\rightarrow+\infty$}.
\label{eq:vanishingGammazeroInf}%
\end{equation}
Moreover, let $I=(i_{1},\ldots,i_{p})\in\{1,\ldots,m\}^{p}$ \emph{(}with
$p\geq1$\emph{)} be a given multi-index. Then, there exists $c>0$, only
depending on $I$, such that
\begin{equation}
\left\vert X_{I}\Gamma_{0}(x)\right\vert \leq c\Vert x\Vert^{\nu-|I|-Q}\text{
\ \ for every $x\in\mathbb{R}^{n},\,x\neq0$}. \label{eq:XIGammazeroEstimate}%
\end{equation}
In particular, we also have
\begin{equation}
X_{I}\Gamma_{0}(x)\rightarrow0\ \ \ \text{as $\Vert x\Vert\rightarrow+\infty$%
}. \label{eq:vanishingXIGammazeroInf}%
\end{equation}

\item[(3)] (Local integrability of $\Gamma_{0}$). We have
\[
\Gamma_{0}\in L_{\mathrm{loc}}^{1}(\mathbb{R}^{n}).
\]
More generally, for every multi-index $I=(i_{1},\ldots,i_{p})\in
\{1,\ldots,m\}^{p}$ \emph{(}with $p\geq1$\emph{)} such that $|I|<\nu$, we
have
\[
X_{I}\Gamma_{0}\in L_{\mathrm{loc}}^{1}(\mathbb{R}^{n}).
\]

\item[(4)] ($\Gamma_{0}$ left-inverts $L$). Let $y\in\mathbb{R}^{n}$ be fixed,
and let
\[
\Gamma(x,y)=\Gamma_{0}(y^{-1}\star x)\qquad(x\neq y).
\]
Then, for every $\varphi\in C_{0}^{\infty}(\mathbb{R}^{n})$ we have
\begin{equation}
\int_{\mathbb{R}^{n}}\Gamma(x,y)L^{\ast}\varphi(x)\,dx=-\varphi(y).
\label{eq:GammazeroleftinverseG}%
\end{equation}

\item[(5)] ($\Gamma_{0}$ right-inverts $L$). Let $\varphi\in C_{0}^{\infty
}(\mathbb{R}^{n})$ be fixed, and let
\begin{equation}
\Lambda_{\varphi}(x)=\int_{\mathbb{R}^{n}}\Gamma(x,y)\varphi(y)\,dy=\int%
_{\mathbb{R}^{n}}\Gamma_{0}(y^{-1}\star x)\varphi(y)\,dy.
\label{eq:defLambdaphi}%
\end{equation}

Then, the following assertions hold.
\begin{align*}
\mathrm{(i)}\,\,  &  \text{$\Lambda_{\varphi}\in C^{\infty}(\mathbb{R}^{n})$,
and $\Lambda_{\varphi}(x)\rightarrow0$ as $\Vert x\Vert\rightarrow+\infty$};\\
\mathrm{(ii)}\,\,  &  \text{$L(\Lambda_{\varphi})=-\varphi$ pointwise in
$\mathbb{R}^{n}$.}%
\end{align*}

\item[(6)] (Relation between $\Gamma_{0}$ and $\Gamma_{0}^{\ast}$). Let
$\Gamma_{0}^{\ast}$ be the global fundamental solution of the formal transpose
$L^{\ast}$ of $L$, see Definition \ref{Def formal transposed}. Then, we have
\begin{equation}
\Gamma_{0}^{\ast}(x)=\Gamma_{0}(x^{-1})\quad\text{for every $x\in
\mathbb{R}^{n}\setminus\{0\}$} \label{eq:GammazeroGammazerostarEqual}%
\end{equation}
In particular, assertions \emph{(1)}-\emph{(5)} hold also for $\Gamma
_{0}^{\ast}$.
\end{enumerate}

Here, $\Vert\cdot\Vert$ is \emph{any} fixed homogeneous norm on $\mathbb{G}$.
\end{theorem}

This result is due to Folland, \cite{Fo2}. More precisely: existence and
uniqueness of $\Gamma_{0}$ and point (1) are contained in \cite[Thm.\,2.1.]%
{Fo2}; point (6) is the Remark after \cite[Thm.\,2.1.]{Fo2}. Point (2) is an
easy consequence of point (1). Point (3) follows from point (2) in view of the
integrability properties discussed in Remark \ref{rem:HomNormLloc}. Points (4)
and (5) are substantially contained in \cite[Cor.\,2.8]{Fo2}, where it is
proved that $\Gamma_{0}$ right-inverts and left-inverts $L$ in the
distributional sense. Then point (ii) follows from (i) and the fact that
$\mathcal{L}(\Lambda_{\varphi})=-\varphi$ in the sense of distributions.

\begin{remark}
We stress the fact that the above theorem does not depend on the explicit form
of the operator $L$. In particular, we can apply it to Rockland operators of
type (\ref{power of Rockland}) as well as to powers of H\"{o}rmander operators
(\ref{power of Hormander}), whenever the required relation between the
homogeneity degree of $L$ and the homogeneous dimension $Q$ is satisfied.
\end{remark}

\section{Properties of generalized Rockland operators}

\label{sec:PropGenRockland}

\subsection{Lifting to homogeneous groups\label{sec lifting}}

Taking into account the results recalled in the previous section for
homogeneous and \emph{left-invariant} operators on a homogeneous Lie group
$\mathbb{G}$, we can now begin the study of generalized Rockland operators.

To begin with, we recall that, if $X_{1},\ldots,X_{m}$ satisfy Assump\-tion
\textbf{(H1)}, we know from Remarks \ref{rem:HormanderRn} and
\ref{rem:ConseguenzeH1H2} that

\begin{enumerate}
\item the $X_{i}$'s satisfy H\"{o}rmander's condition at every point of
$\mathbb{R}^{n}$;

\item the Lie algebra $\mathrm{Lie}(X)$ generated by $X = \{X_{1},\ldots
,X_{m}\}$ has \emph{finite dimension}, say $N\in\mathbb{N}$, and we have
$N\geq n$.
\end{enumerate}

Furthermore, as noted in \cite[Rem.\,1.2]{BBB-fundsol}, the case
\[
N=n
\]
occurs \emph{if and only if} there exists a homogeneous Lie group
$\mathbb{G}=(\mathbb{R}^{n},\star,\delta_{\lambda})$ (in the sense of
De\-fi\-nition \ref{Def homogeneous group}, with $\{\delta_{\lambda
}\}_{\lambda}$ as in \eqref{dilations}) such that
\[
\mathrm{Lie}(\mathbb{G})=\mathrm{Lie}(X),
\]
that is, if and only if the $X_{i}$'s are left-invariant on some homogeneous
Lie group on $\mathbb{R}^{n}$. However, in that situation the known results
recalled in Section \ref{sec left-invariant} apply. Therefore, from now on we
tacitly understand that
\begin{equation}%
\begin{split}
(i)\,\,  &  \text{$X_{1},\ldots,X_{m}$ satisfy Assumption \textbf{(H1)}%
;}\\[0.1cm]
(ii)\,\,  &  p=N-n=\mathrm{dim}(\mathrm{Lie}(X))-n\geq1;
\end{split}
\label{eq:pgeqone}%
\end{equation}
Moreover, according to $(ii)$, we denote the points $z\in\mathbb{R}%
^{N}=\mathbb{R}^{n}\times\mathbb{R}^{p}$ by%
\[
z=(x,\xi),\quad\text{with $x\in\mathbb{R}^{n}$ and $\xi\in\mathbb{R}^{p}$}.
\]
Thanks to assumption \eqref{eq:pgeqone}-$(ii)$ (which, as discussed above, is
\emph{non-restric\-ti\-ve}), we can then state the following (crucial)
\emph{lifting property} of the $X_{i}$'s.

\begin{theorem}
[{Lifting property, see {\cite[Thm.\,3.2]{BB-lift}}}]\label{thm:Adrift}Given
in $\mathbb{R}^{n}$ the sy\-s\-tem
\[
X = \{X_{1},...,X_{m}\}
\]
satisfying Assumption (H1), there exists a \emph{ho\-mo\-ge\-neous group
$\mathbb{G}=(\mathbb{R}^{N},\star,D_{\lambda})$} (with $N$ as above), of
homogeneous dimension $Q>q$ \emph{(}where $q$ is as in Assumption
\textbf{\emph{(H1)}}\emph{)}, and there exists a system
\[
\widetilde{X}=\{\widetilde{X}_{1},\ldots,\widetilde{X}_{m}\}\subseteq
\mathrm{Lie}(\mathbb{G})
\]
satisfying the following properties:

\begin{enumerate}
\item $\mathrm{Lie}(\widetilde{X})=\mathrm{Lie}(\mathbb{G})$;

\item $\widetilde{X}_{i}$ is $D_{\lambda}$-homogeneous of degree $\nu_{i}$;

\item for every $1\leq j\leq m$, we have
\begin{equation}
\label{eq:tildeXiliftXi}\widetilde{X}_{j} = X_{j}+R_{j}(x,\xi),
\end{equation}
where $R_{j}(x,\xi)$ is a \emph{non-vanishing} smooth vector field operating
only in the $\xi\in\mathbb{R}^{p}$ varia\-ble, with coefficients possibly
depending on $(x,\xi)$.
\end{enumerate}

Finally, the family of dilations $\{D_{\lambda}\}_{\lambda}$ takes the
following \emph{lifted} form
\begin{equation}
D_{\lambda}(x,\xi)=(\delta_{\lambda}(x),E_{\lambda}(\xi)),
\label{eq:Dlambdalift}%
\end{equation}
where $\delta_{\lambda}$ is as in Assumption \textbf{(H1)}, and $E_{\lambda
}(\xi)=(\lambda^{\tau_{1}}\xi_{1},\ldots,\lambda^{\tau_{p}}\xi_{p})$ for
suitable non-negative integers $1\leq\tau_{1}\leq\ldots\leq\tau_{p}$. Hence,
we have
\begin{equation}
\textstyle Q=\sum_{i=1}^{n}\sigma_{i}+\sum_{j=1}^{p}\tau_{j}\equiv
q+\mathcal{E}. \label{eq:dimqGruppiOm}%
\end{equation}

\end{theorem}

\begin{remark}
\label{rem:CambiDiVariabile}As a matter of fact, Theorem \ref{thm:Adrift} is
established in \cite{BB-lift} in the particular case in which all the $X_{i}%
$'s are \emph{$\delta_{\lambda}$-homogeneous of degree $1$} \emph{(}%
na\-me\-ly, when $\nu_{1}=\ldots=\nu_{m}=1$\emph{)}; however, it is not
difficult to see that the proof given in \cite{BB-lift} can be easily extended
to the general case considered here. \vspace{0.1cm}

Moreover, from the \emph{explicit} construction of $\mathbb{G}$, one can
derive the following useful facts (see \cite[Rem.\,8]{BB-lift} and
\cite[Lem.\,4.5]{BBB-fundsol}).

\begin{enumerate}
\item Let $x,y\in\mathbb{R}^{n}$ be arbitrarily fixed, and let
\[
\Psi_{x,y}:\mathbb{R}^{p}\rightarrow\mathbb{R}^{p},\qquad\Psi_{x,y}(\xi
)=\pi_{p}\big((y,0)^{-1}\star(x,\xi)\big),
\]
where $\pi_{p}:\mathbb{R}^{N}\rightarrow\mathbb{R}^{p}$ denotes the projection
of $\mathbb{R}^{N}$ onto $\mathbb{R}^{p}$. Then, this map $\Psi_{x,y}$ is a
\emph{smooth diffeomorphism} of $\mathbb{R}^{p}$ onto itself, smooth\-ly
depending on the fixed $x,y$, whose Jacobian determinant is $\pm1$.

In other words, the \emph{change of variable} $\xi=\Psi_{x,y}^{-1}(\zeta)$
satisfies
\[
\pi_{p}\big((y,0)^{-1}\star(x,\xi)\big)=\zeta\quad\text{and}\quad d\xi
=d\zeta.
\]

\item Let $x,y\in\mathbb{R}^{n}$ be arbitrarily fixed, and let
\[
\Phi_{x,y}:\mathbb{R}^{p}\rightarrow\mathbb{R}^{p},\qquad\Phi_{x,y}(\xi
)=\pi_{p}\big((y,0)\star(y,\xi)^{-1}\star(x,0)\big).
\]
Then, this map $\Phi_{x,y}$ is a \emph{smooth diffeomorphism} of
$\mathbb{R}^{p}$ onto itself, smooth\-ly depending on the fixed $x,y$, whose
Jacobian determinant is $\pm1$. Moreo\-ver, the following identity holds
\[
(y,0)^{-1}\star(x,\Phi_{x,y}(\zeta))=(y,\zeta)^{-1}\star(x,0).
\]
In other words, the \emph{change of variable} $\xi=\Phi_{x,y}(\zeta)$
satisfies
\begin{equation}
(y,0)^{-1}\star(x,\xi)=(y,\zeta)^{-1}\star(x,0)\quad\text{and}\quad
d\xi=d\zeta. \label{eq:CVPhiIdentity}%
\end{equation}

\end{enumerate}

We will repeatedly exploit these facts in the sequel.
\end{remark}

The \emph{global lifting property} in Theorem \ref{thm:Adrift} is the
\emph{key ingredient} in our investi\-gati\-on on general higher-order
homogeneous operators of the form
\begin{equation}
\label{eq:GeneralLGlobalProperties}\mathcal{L} = \sum_{|I| = \nu}c_{I}
X_{I}=\sum
_{\begin{subarray}{c} I = (i_1,\ldots,i_k) \\ \nu_{i_1}+\cdots+\nu_{i_k} = \nu \end{subarray}}%
c_{I} X_{i_{1}}\cdots X_{i_{k}}%
\end{equation}
(where $c_{I}\in\mathbb{R}$ and $\nu> 0$ is a fixed positive integer).
\vspace{0.05cm}

Indeed, following the notation of Theorem \ref{thm:Adrift} (used here and in
the sequel), we observe that the operator, defined on the group $\mathbb{G}%
=\mathbb{R}^{N}$,
\begin{equation}
\widetilde{\mathcal{L}}=\sum_{|I|=\nu}c_{I}\widetilde{X}_{I}=\sum
_{\begin{subarray}{c} I = (i_1,\ldots,i_k) \\ \nu_{i_1}+\cdots+\nu_{i_k} = \nu \end{subarray}}%
c_{I}\widetilde{X}_{i_{1}}\cdots\widetilde{X}_{i_{k}} \label{eq:tildeLgeneral}%
\end{equation}
is \emph{homogeneous of degree $\nu$} (with respect to the family of dilations
$\{D_{\lambda}\}_{\lambda}$) and also \emph{left-invariant} (with respect to
$\star$); thus, if we assume that
\begin{equation}
\text{both $\widetilde{\mathcal{L}}$ and $(\widetilde{\mathcal{L}})^{\ast}$
are hypoelliptic}, \label{eq:AssLLstartildeHypo}%
\end{equation}
we may apply to $\widetilde{\mathcal{L}}$ all the results recalled in the
previous section, thereby ob\-taining several \emph{global properties}
(existence of a well-behaved global fundamental solution, validity of
Liouville-type theorems, etc.).

On the other hand, owing to \eqref{eq:tildeXiliftXi}, we see that the operator
$\widetilde{\mathcal{L}}$ \emph{lifts} $\mathcal{L}$ in the following sense:
denoting by $\pi_{n}$ the projection of $\mathbb{R}^{N}=\mathbb{R}^{n}%
\times\mathbb{R}^{p}$ onto $\mathbb{R}^{n}$, for every smooth function $u\in
C^{\infty}(\mathbb{R}^{n})$ we have
\begin{equation}
\widetilde{\mathcal{L}}(u\circ\pi_{n})=(\mathcal{L}u)\circ\pi_{n};
\label{eq:LhtildeliftsLh}%
\end{equation}
on account of \eqref{eq:LhtildeliftsLh}, it is therefore natural to attempt to
derive analogous \emph{global proper\-ties} for $\mathcal{L}$ by means of a
\emph{saturation argument}. This type of argument has already proved effective
in \cite{BB-lift}, where the existence of a well-behaved global funda\-mental
solution for the sum of squares $\mathcal{L}=\sum_{i=1}^{m}X_{i}^{2}$ has been established.

We can now also rephrase Definition \ref{Def gen Rockland 1} as follows:

\begin{definition}
[Generalized Rockland operators, second formulation]\label{Def gen Rockland 2}%
We say that an operator $\mathcal{L}$ of kind
(\ref{eq:GeneralLGlobalProperties}) is a generalized Rockland operator if
$X_{1},...,X_{m}$ satisfy (\ref{eq:pgeqone}) and the lifted operator
$\widetilde{\mathcal{L}}$ in (\ref{eq:tildeLgeneral}) is hypoelliptic on the
lifted space $\mathbb{R}^{N}$.
\end{definition}

Since the lifted vector fields $\widetilde{X}_{i}$ are particular homogeneous
left-invariant vector fields, if $\mathcal{L}$ satisfies Definition
\ref{Def gen Rockland 1}, in particular the lifted operator
$\widetilde{\mathcal{L}}$ will be hypoelliptic. Therefore the present
Definition \ref{Def gen Rockland 2} is actually more general than Definition
\ref{Def gen Rockland 1}.

\subsection{Hypoellipticity of $\mathcal{L}$\label{sec hypoellipticity}}

We begin by showing how Theorem \ref{thm:Adrift} can be used to prove the
\emph{hypoellipticity of any operator $\mathcal{L}$} of the form
\eqref{eq:GeneralLGlobalProperties}, provided that the same property holds for
the \emph{lifted operator $\widetilde{\mathcal{L}}$}. \medskip

\begin{proof}
[Proof of Theorem \ref{Thm main hypoellipticity}]Let $\Omega\subseteq
\mathbb{R}^{n}$ be an open set, and let $\Lambda\in\mathcal{D}^{\prime}%
(\Omega)$ be a di\-stri\-bution on $\Omega$. We assume that there exists $f\in
C^{\infty}(\Omega)$ such that
\[
\mathcal{L}\Lambda=f\quad\text{in $\mathcal{D}^{\prime}(\Omega)$},
\]
that is,
\[
\left\langle \mathcal{L}\Lambda,\varphi\right\rangle =\int_{\Omega}%
f\varphi\,dx\quad\forall\,\,\varphi\in C_{0}^{\infty}(\Omega).
\]
We then define $\widetilde{\Lambda}=\Lambda\otimes1\in\mathcal{D}^{\prime
}(\Omega\times\mathbb{R}^{p})$ (where $1$ denotes the distribution on
$\mathbb{R}^{p}$ associated with the locally-integrable function $u\equiv1$,
and $\otimes$ is the usual tensor pro\-duct between distributions), and we
claim that
\begin{equation}
\widetilde{\mathcal{L}}\widetilde{\Lambda}=f\quad\text{in $\mathcal{D}%
^{\prime}(\Omega\times\mathbb{R}^{p})$} \label{eq:ClaimLtildeLambdatilde}%
\end{equation}
To prove this claim, we exploit an \emph{approximation argument}. Let
$\{\theta_{k}\}_{k}\subseteq C^{\infty}(\Omega)$ be a sequence of \emph{smooth
functions} in $\Omega$ such that
\[
\text{$\theta_{k}\rightarrow\Lambda$ in $\mathcal{D}^{\prime}(\Omega)$ (as
$k\rightarrow+\infty$)}.
\]
Then, recalling that $\widetilde{\Lambda}=\Lambda\otimes1$ is the distribution
on $\Omega\times\mathbb{R}^{p}$ defined by
\begin{equation}%
\begin{split}
\langle\widetilde{\Lambda},\phi\rangle &  =\left\langle \Lambda,x\mapsto
\langle1,\xi\mapsto\phi(x,\xi)\rangle\right\rangle \\
&  =\left\langle \Lambda,\int_{\mathbb{R}^{p}}\phi(\cdot,\xi)\,d\xi
\right\rangle \quad\forall\,\,\phi\in C_{0}^{\infty}(\Omega\times
\mathbb{R}^{p}),
\end{split}
\label{eq:defTensorProd}%
\end{equation}
and since the operator $\widetilde{\mathcal{L}}$ \emph{is a lifting} of
$\mathcal{L}$, see \eqref{eq:LhtildeliftsLh}, for every $\phi\in C_{0}%
^{\infty}(\Omega\times\mathbb{R}^{p})$ we have the following computation
\begin{align*}
\langle\widetilde{\mathcal{L}}\widetilde{\Lambda},\phi\rangle &
=\langle\widetilde{\Lambda},\widetilde{\mathcal{L}}^{\ast}\phi\rangle
=\left\langle \Lambda,x\mapsto\int_{\mathbb{R}^{p}}(\widetilde{\mathcal{L}%
}^{\ast}\phi)(x,\xi)\,d\xi\right\rangle \\
&  =\lim_{k\rightarrow+\infty}\int_{\Omega}\theta_{k}(x)\left(  \int%
_{\mathbb{R}^{p}}(\widetilde{\mathcal{L}}^{\ast}\phi)(x,\xi)\,d\xi\right)
dx\\
&  =\lim_{k\rightarrow+\infty}\int_{\Omega\times\mathbb{R}^{p}}\theta
_{k}(x)(\widetilde{\mathcal{L}}^{\ast}\phi)(x,\xi)\,dx\,d\xi\\
&  (\text{integrating by parts, and using \eqref{eq:LhtildeliftsLh}})\\
&  =\lim_{k\rightarrow+\infty}\int_{\Omega\times\mathbb{R}^{p}}(\mathcal{L}%
\theta_{k})(x)\phi(x,\xi)\,dx\,d\xi\\
&  =\lim_{k\rightarrow+\infty}\int_{\Omega}(\mathcal{L}\theta_{k})(x)\left(
\int_{\mathbb{R}^{p}}\phi(x,\xi)\,d\xi\right)  dx\\
&  =\left\langle \mathcal{L}\Lambda,\int_{\mathbb{R}^{p}}\phi(\cdot,\xi
)\,d\xi\right\rangle .
\end{align*}
From this, since $\mathcal{L}\Lambda=f$ in $\mathcal{D}^{\prime}(\Omega)$, we
obtain
\[
\langle\widetilde{\mathcal{L}}\widetilde{\Lambda},\phi\rangle=\int_{\Omega
}f(x)\left(  \int_{\mathbb{R}^{p}}\phi(x,\xi)\,d\xi\right)  dx=\int%
_{\Omega\times\mathbb{R}^{p}}f(x)\phi(x,\xi)\,dx\,d\xi,
\]
and this proves the claimed \eqref{eq:ClaimLtildeLambdatilde} (by the
arbitrariness of $\phi$). \vspace{0.1cm}

Now we have established \eqref{eq:ClaimLtildeLambdatilde}, we can easily
complete the proof of the theorem. Indeed, since $f$ can be thought of as a
smooth function on $\Omega\times\mathbb{R}^{p}$, and since we are assuming
that the operator $\widetilde{\mathcal{L}}$ is hypoelliptic on every open
subset of $\mathbb{R}^{N}$, there exists $U\in C^{\infty}(\Omega
\times\mathbb{R}^{p})$ such that
\[
\text{$\widetilde{\Lambda}=U$ in $\mathcal{D}^{\prime}(\Omega\times
\mathbb{R}^{p})$}.
\]
As a consequence, \emph{given any $\varphi\in C_{0}^{\infty}(\Omega)$}, and
choosing (once and for all) a test function $\psi_{0}\in C_{0}^{\infty
}(\mathbb{R}^{p})$ satisfying $0\leq\psi_{0}\leq1$ on $\mathbb{R}^{p}$ and
\[
\int_{\mathbb{R}^{p}}\psi_{0}\,d\xi=1,
\]
we obtain (see \eqref{eq:defTensorProd} with $\phi=\varphi\psi_{0}\in
C_{0}^{\infty}(\Omega\times\mathbb{R}^{p})$)
\begin{align*}
\langle\Lambda,\varphi\rangle &  =\left\langle \Lambda,\varphi\cdot\left(
\int_{\mathbb{R}^{p}}\psi_{0}(\xi)\,d\xi\right)  \right\rangle =\langle
\widetilde{\Lambda},\varphi\psi_{0}\rangle\\
&  =\int_{\Omega\times\mathbb{R}^{p}}U(x,\xi)\varphi(x)\psi_{0}(\xi
)\,dx\,d\xi\\
&  =\int_{\Omega}\left(  \int_{\mathbb{R}^{p}}U(x,\xi)\psi_{0}(\xi
)\,d\xi\right)  \varphi(x)\,dx,
\end{align*}
that is,
\[
\text{$\Lambda=u$ in $\mathcal{D}^{\prime}(\Omega)$},\qquad\text{where
$u(x)=\int_{\mathbb{R}^{p}}U(x,\xi)\psi_{0}(\xi)\,d\xi$}.
\]
Since $u\in C^{\infty}(\Omega)$, we conclude that $\mathcal{L}$ is $C^{\infty
}$-hypoelliptic in $\Omega$.
\end{proof}

\subsection{A Liouville-type theorem for $\mathcal{L}$\label{subsec Liouville}%
}

We now turn to proving that Theorem \ref{thm:Adrift} also allows us to derive
a \emph{Liouville-type property} for $\mathcal{L}$, analogous to the one
stated in Theorem \ref{thm:LiouvilleGroups} for homogeneous operators that are
\emph{left-invariant} on some homogeneous group. As noted in Remark
\ref{Remark Liouville}, this result could also be deduced by applying a
general result on homogeneous hypoelliptic operators proved in \cite{L}.
However, we think that the following short alternative proof of this result
has an independent interest.

\bigskip

\begin{proof}
[Proof of Theorem \ref{Thm main Liouville}]We proceed essentially as in the
proof of Theorem \ref{Thm main hypoellipticity}. To begin with, we consider
the distribution on $\mathbb{R}^{N}=\mathbb{R}^{n}\times\mathbb{R}^{p}$ given
by
\[
\widetilde{\Lambda}=\Lambda\otimes1
\]
We then observe that, since both $\Lambda$ and $1$ are tempered distributions
(on $\mathbb{R}^{n}$ and $\mathbb{R}^{p}$, respectively), then
$\widetilde{\Lambda}\in\mathcal{S}^{\prime}(\mathbb{R}^{N})$; moreover, since
\[
\text{$\mathcal{L}\Lambda=0$ in $\mathcal{S}^{\prime}(\mathbb{R}^{n})$},
\]
an argument \emph{entirely analogous} to the one exploited in the proof of
Theorem~\ref{Thm main hypoellipticity} (the only difference being that here
the test functions belong to $\mathcal{S}(\mathbb{R}^{N})$ rather than to
$C_{0}^{\infty}(\mathbb{R}^{n})$) shows that
\[
\widetilde{\mathcal{L}}\widetilde{\Lambda}=0\quad\text{in $\mathcal{S}%
^{\prime}(\mathbb{R}^{N})$}.
\]
From this, since we are assuming \eqref{eq:AssLLstartildeHypo}, we may apply
Theorem \ref{thm:LiouvilleGroups}, ensuring the existence of a polynomial
function $P=P(x,\xi)$ in $\mathbb{R}^{N}$ such that
\begin{equation}
\text{$\widetilde{\Lambda}=P$ in $\mathcal{S}^{\prime}(\mathbb{R}^{N})$}.
\label{eq:tildeLambdaPolynomial}%
\end{equation}
With \eqref{eq:tildeLambdaPolynomial} at hand, we can easily complete the
proof of the theorem. Indeed, \emph{given any $\varphi\in\mathcal{S}%
(\mathbb{R}^{n})$}, and choosing once and for all $\psi_{0}\in\mathcal{S}%
(\mathbb{R}^{p})$ such that
\[
\text{$0\leq\psi_{0}\leq1$ on $\mathbb{R}^{p}$}\quad\text{and}\quad
\int_{\mathbb{R}^{p}}\psi_{0}\,d\xi=1,
\]
we obtain, by \eqref{eq:defTensorProd},
\begin{align*}
\langle\Lambda,\varphi\rangle &  =\left\langle \Lambda,\varphi\cdot\left(
\int_{\mathbb{R}^{p}}\psi_{0}(\xi)\,d\xi\right)  \right\rangle =\langle
\widetilde{\Lambda},\varphi\psi_{0}\rangle\\
&  =\int_{\mathbb{R}^{n}\times\mathbb{R}^{p}}P(x,\xi)\varphi(x)\psi_{0}%
(\xi)\,dx\,d\xi\\
&  =\int_{\mathbb{R}^{n}}\left(  \int_{\mathbb{R}^{p}}P(x,\xi)\psi_{0}%
(\xi)\,d\xi\right)  \varphi(x)\,dx,
\end{align*}
that is,
\[
\text{$\Lambda=p$ in $\mathcal{S}^{\prime}(\mathbb{R}^{n})$},\qquad\text{where
$p(x)=\int_{\mathbb{R}^{p}}P(x,\xi)\psi_{0}(\xi)\,d\xi$}.
\]
Since $p\in C^{\infty}(\mathbb{R}^{n})$ is a polynomial function (as the same
is true of $P$), we obtain the desired \eqref{eq:LambdapolinomioGeneral}, and
the proof is complete.
\end{proof}

\subsection{Existence of a global fundamental
solution\label{sec existence fund sol}}

Finally, we turn to showing that a \emph{lifting-plus-saturation argument}
arising from Theorem \ref{thm:Adrift} allows us to pro\-ve the existence of a
well-behaved global fundamental solution for $\mathcal{L}$. \vspace{0.1cm}

Roughly put, our argument goes as follows. First of all, since $\widetilde{X}%
_{1},\ldots,\widetilde{X}_{m}$ are \emph{left-invariant} (with respect to the
ho\-mo\-geneous gro\-up $\mathbb{G}=(\mathbb{R}^{N},\star,D_{\lambda})$) and
\emph{homogeneous} (with respect to the family of dilations $\{D_{\lambda
}\}_{\lambda}$ on the higher-di\-men\-sional space $\mathbb{R}^{N}$), the
\emph{lifted operator} $\widetilde{\mathcal{L}}$ defined in
\eqref{eq:tildeLgeneral}, that is,
\[
\widetilde{\mathcal{L}}=\sum_{|I|=\nu}c_{I}\widetilde{X}_{I}=\sum
_{\begin{subarray}{c}
I = (i_1,\ldots,i_k)	\\
\nu_{i_1}+\cdots+\nu_{i_k} = \nu
\end{subarray}}c_{I}\widetilde{X}_{i_{1}}\cdots\widetilde{X}_{i_{k}}%
\]
is \emph{$\nu$-homogeneous and left-invariant on $\mathbb{G}$}. Thus, assuming
from now on that
\begin{equation}
\nu<q\,\,(<Q) \label{eq:nulessq}%
\end{equation}
(recall that $Q$ is the homogeneous dimension of the group $\mathbb{G}$), we
may apply Theorem \ref{thm:PropertiesGammGroups}, which ensures that the
operator $\widetilde{\mathcal{L}}$ possesses a \emph{well-behaved global
fundamental solution} of the form
\begin{equation}
\widetilde{\Gamma}((x,\xi),(y,\eta))=\widetilde{\Gamma}_{0}((y,\eta)^{-1}%
\star(x,\xi)), \label{eq:widetildeGammaExist}%
\end{equation}
where $\widetilde{\Gamma}_{0}\in C^{\infty}(\mathbb{R}^{N}\setminus\{0\})$
satisfies the following properties:

\begin{enumerate}
\item[(a)] $\widetilde{\Gamma}_{0}$ is $D_{\lambda}$-homogeneous of degree
$\nu-Q$;

\item[(b)] for every $\varphi\in C_{0}^{\infty}(\mathbb{R}^{N})$ we have
\[
\int_{\mathbb{R}^{N}}\widetilde{\Gamma}_{0}(x,\xi) (\widetilde{\mathcal{L}%
})^{*}\varphi(x,\xi)\,dx\,d\xi= -\varphi(0).
\]

\end{enumerate}

Recalling that $\widetilde{\mathcal{L}}$ is a lifting of $\mathcal{L}$, see
\eqref{eq:LhtildeliftsLh}, it is then natural to expect that the ``integral
saturation''
\begin{equation}
\label{eq:ideaGammaGammatilde}\Gamma(x,y) = \int_{\mathbb{R}^{p}%
}\widetilde{\Gamma}((x,\xi),(y,0))\,d\xi= \int_{\mathbb{R}^{p}}%
\widetilde{\Gamma}_{0}((y,0)^{-1}\star(x,\xi))\,d\xi
\end{equation}
provides a \emph{global fundamental solution for $\mathcal{L}$} with pole at
$y$ (see \cite{BB-lift}). \vspace{0.1cm}

As a matter of fact, the above intuition can be made rigorous with the help of
\cite[Thm.\,2.1]{BB-lift}, which ensures that the function $\Gamma$ defined in
\eqref{eq:ideaGammaGammatilde} is \emph{in\-deed} a global fundamental
solution of $\mathcal{L}$, provided that the `lifted' operator
$\widetilde{\mathcal{L}}$ and its (global) fundamental solution
$\widetilde{\Gamma}$ satisfy the properties contained in the next two Propositions.

\begin{proposition}
[Saturable lifting]\label{prop:S1S2vere}The operator $\widetilde{\mathcal{L}}$
is a \emph{saturable lifting} of $\mathcal{L}$ in the sense of \cite[Def.\,2]%
{BB-lift}, that is, the following properties hold.

\begin{itemize}
\item[(S1)] Setting $\mathcal{R}=\widetilde{\mathcal{L}}-\mathcal{L}$, every
summand of the \emph{formal adjoint of $\mathcal{R}$} operates at least once
in the $\xi$ variables, that is, we have
\begin{equation}
\mathcal{R}^{\ast}=\sum_{\beta\neq0}r_{\alpha,\beta}(x,\xi)\partial
_{x}^{\alpha}\partial_{\xi}^{\beta}, \label{eq:RadjointEsplicitar}%
\end{equation}
for suitable smooth functions $r_{\alpha,\beta}$ (not all identically zero).

\item[(S2)] There exists a sequence $\{\theta_{j}\}_{j}\subseteq C_{0}%
^{\infty}(\mathbb{R}^{p})$ such that

\item[i)] $0\leq\theta_{j}\leq1$ on $\mathbb{R}^{p}$ for every $j\in
\mathbb{N}$;

\item[ii)] setting $\Omega_{j}=\{\xi\in\mathbb{R}^{p}:\,\theta_{j}(\xi)=1\}$,
we have
\[
\Omega_{j}\Subset\Omega_{j+1}\quad\text{and}\quad\bigcup_{j=1}^{+\infty}%
\Omega_{j}=\mathbb{R}^{p};
\]

\item[iii)] for every compact set $K\subseteq\mathbb{R}^{n}$ and for every
coefficient function $r_{\alpha,\beta}$ of $\mathcal{R}^{\ast}$ as in
\eqref{eq:RadjointEsplicitar} one can find a constant $C_{\alpha,\beta}(K)>0$
such that
\begin{equation}
\begin{gathered} \big|r_{\alpha,\beta}(x,\xi)\partial_\xi^\beta\theta_j(\xi)\big| \leq C_{\alpha,\beta}(K) \\[0.1cm] \text{for every $x\in K,\,\xi\in\mathbb{R}^p$ and $j\in\mathbb{N}$}. \end{gathered} \label{eq:assumptionS2estim}%
\end{equation}

\end{itemize}
\end{proposition}

\begin{proposition}
[Integrability properties of $\widetilde{\Gamma}$]\label{prop:P1P2}Let
$\widetilde{\Gamma}$ be the fundamental soluti\-on of $\widetilde{\mathcal{L}%
}$ defined in \eqref{eq:widetildeGammaExist}. Then, the following properties hold:

\begin{itemize}
\item[(P1)] for every fixed $x\neq y\in\mathbb{R}^{n}$, we have
\[
\xi\mapsto\widetilde{\Gamma}((x,\xi),(y,0))=\widetilde{\Gamma}_{0}%
((y,0)^{-1}\star(x,\xi))\in L^{1}(\mathbb{R}^{p});
\]

\item[(P2)] for every $y\in\mathbb{R}^{n}$ and every compact set
$K\subseteq\mathbb{R}^{n}$, we have
\[
(x,\xi)\mapsto\widetilde{\Gamma}((x,\xi),(y,0))=\widetilde{\Gamma}%
_{0}((y,0)^{-1}\star(x,\xi))\in L^{1}(K\times\mathbb{R}^{p}).
\]

\end{itemize}
\end{proposition}

Before proceeding, we need to introduce the following notation, which will be
used in the following.

Suppose we are given a family of non-isotropic dilations of the form%
\[
d_{\lambda}:\mathbb{R}^{k}\rightarrow\mathbb{R}^{k},\qquad d_{\lambda
}(v)=(\lambda^{\varepsilon_{1}}v_{1},\ldots,\lambda^{\varepsilon_{k}}v_{k}),
\]
on some Euclidean space $\mathbb{R}^{k}$ (with $k\geq1$ and $1\leq
\varepsilon_{1}\leq\cdots\leq\varepsilon_{k}$). Then, for every multi-index
$\alpha=(\alpha_{1},\ldots,\alpha_{k})$ (with $\alpha_{i}$ nonnegative
integers), we set
\[
\mathcal{H}(\alpha)=\sum_{j=1}^{k}\alpha_{j}\qquad\text{and}\qquad
\mathcal{H}_{d}(\alpha)=\sum_{j=1}^{k}\alpha_{j}\varepsilon_{j}.
\]

Moreover, we denote by $\varrho_{d}$ the \emph{homogeneous norm} associated
with the family $\{d_{\lambda}\}_{\lambda>0}$, namely
\begin{equation}
\varrho_{d}(v)=\sum_{i=1}^{k}|v_{i}|^{1/\varepsilon_{i}}. \label{eq:varrhod}%
\end{equation}

We will make use of these notations for the three families of dilations that
natu\-rally arise from Assumption \textbf{(H1)} and Theorem \ref{thm:Adrift},
namely: \vspace{0.1cm}

i)\,\,the family $\{\delta_{\lambda}\}_{\lambda}$, which is defined on
$\mathbb{R}^{n}$; \vspace{0.05cm}

ii)\,\,the family $\{E_{\lambda}\}_{\lambda}$, which is defined on
$\mathbb{R}^{p}$; \vspace{0.05cm}

iii)\thinspace\thinspace the family $\{D_{\lambda}=(\delta_{\lambda
},E_{\lambda})\}_{\lambda}$, which is defined on $\mathbb{R}^{N}%
=\mathbb{R}^{n}\times\mathbb{R}^{p}$. \medskip

The reader will note that we prefer to use the symbol $\varrho_{d}$, different
from $\left\Vert \cdot\right\Vert $ which is used on homogeneous groups, to
avoid confusion.

\medskip

Let us now turn to prove Propositions \ref{prop:S1S2vere} and \ref{prop:P1P2}.
To this end, we begin by establishing the following technical lemma.

\begin{lemma}
\label{lem:AggiuntoOperaInxi} Setting $\mathcal{R}=\widetilde{\mathcal{L}%
}-\mathcal{L}$, we have
\begin{equation}
\mathcal{R}^{\ast}=(\widetilde{\mathcal{L}})^{\ast}-\mathcal{L}^{\ast}%
=\sum_{\beta\neq0}r_{\alpha,\beta}(x,\xi)\partial_{x}^{\alpha}\partial_{\xi
}^{\beta}, \label{eq:formRstarPerS1}%
\end{equation}
for some \emph{polynomial functions} $r_{\alpha,\beta}$, \emph{not all
identically zero} \emph{(}in other words, every summand of $\mathcal{R}^{*}$
operates at least once in the $\xi$ variables\emph{)}. Furthermore, for each
fixed $\alpha,\beta$, the polynomial $r_{\alpha,\beta}$ has the following
expli\-cit expression
\begin{equation}
r_{\alpha,\beta}(x,\xi)=\sum_{\mathcal{H}_{E}(\gamma)\leq m_{\beta}}%
c_{\alpha,\beta,\gamma}(x)\,\xi^{\gamma}, \label{eq:ralfabetaLemmaTecnico}%
\end{equation}
where $c_{\alpha,\beta,\gamma}$ are polynomial functions only depending on
$x$, and
\begin{equation}
m_{\beta}=\mathcal{H}_{E}(\beta)-1. \label{eq:defHbeta}%
\end{equation}

\end{lemma}

\begin{proof}
First of all, we observe that, by definition,
\begin{align*}
\mathcal{R}^{*}  &  = \sum_{|I| = \nu}c_{I}\big(\widetilde{X}_{I}^{*}%
-X_{I}^{*}\big)\\
&  = \sum
_{\begin{subarray}{c} I = (i_1,\ldots,i_k) \\ \nu_{i_1}+\cdots+\nu_{i_k} = \nu \end{subarray}}%
(-1)^{k}c_{I}\cdot\big(\widetilde{X}_{i_{k}}\cdots\widetilde{X}_{i_{1}}-
X_{i_{k}}\cdots X_{i_{1}}\big);
\end{align*}
as a consequence, to prove the lemma it suffices to show that, \emph{given any
multi-in\-dex $I = (i_{1},\ldots,i_{k})\in\{1,\ldots,m\}^{k}$} (for some
$k\geq1$, and not necessarily satisfy\-ing the condition $|I| = \nu$), we
have
\begin{equation}
\label{eq:ToProveInductionRstarOperatoresxi}\begin{gathered} (\bigstar)\qquad\quad\widetilde{X}_{i_k}\cdots \widetilde{X}_{i_1}- X_{i_k}\cdots X_{i_1} = \sum\nolimits_{\beta\neq 0}r_{\alpha,\beta}(x,\xi)\partial^\alpha_x \partial^\beta_\xi, \\ \begin{array}{c} \text{where each $r_{\alpha,\beta}$ is a polynomial function of the form \eqref{eq:ralfabetaLemmaTecnico}} \\ \text{(possibly depending on $I$), and not all of them vanish.} \end{array} \end{gathered}
\end{equation}
Hence, we turn to proving \eqref{eq:ToProveInductionRstarOperatoresxi} by
\emph{induction on $k$}. \vspace{0.1cm}

If $k = 1$ (that is, $I = (i)$ for $1\leq i\leq m$), then by Theorem
\ref{thm:Adrift} we have
\begin{align*}
\widetilde{X}_{i} - X_{i} = R_{i}(x,\xi),
\end{align*}
where $R_{i}(x,\xi)$ is a \emph{non-vanishing} smooth vector field acting only
on the variables $\xi\in\mathbb{R}^{p}$, with coefficients possibly depending
on $(x,\xi)$, and $D_{\lambda}$-homogeneous of degree $\nu_{i}$ (as the same
is true of $X_{i}$ and $\widetilde{X}_{i}$). Thus,
\begin{equation}
\label{eq:explicitRijStep1}\widetilde{X}_{i}-X_{i} = \sum_{j = 1}^{p}
r_{i,j}(x,\xi)\,\partial_{\xi_{j}} \equiv\sum_{\mathcal{H}(\beta) = 1}%
r_{\beta}(x,\xi) \partial^{\beta}_{\xi},
\end{equation}
where each $r_{i,j} = r_{\beta}$ (with $\beta= e_{j}$ for some $1\leq j\leq
p$) is a polynomial function, $D_{\lambda}$-homoge\-ne\-ous of degree
\[
\tau_{j}-\nu_{i} = \mathcal{H}_{E}(e_{j})-\nu_{i}%
\]
(see Remark \ref{rem:ConseguenzeH1H2}). Hence, $(\bigstar)$ holds in this case
(note that the polynomials $r_{i,j}$ are not all identically zero, since
$R_{i}$ is \emph{non-vanishing}).

As for the \emph{explicit form} of $r_{i,j} = r_{\beta}$ it suffices to
observe that, since $r_{i,j}$ is (smooth and) $D_{\lambda}$-homoge\-ne\-ous of
degree $\mathcal{H}_{E}(e_{j})-\nu_{i}$, we have
\begin{equation}
\label{eq:formrijStep1}%
\begin{split}
r_{i,j}(x,\xi)  &  = \sum_{\mathcal{H}_{\delta}(\gamma_{1})+\mathcal{H}%
_{E}(\gamma_{2}) = \tau_{j}-\nu_{i}}c_{i,j,\gamma_{1},\gamma_{2}}x^{\gamma
_{1}}\xi^{\gamma_{2}}\\
&  = \sum_{\mathcal{H}_{E}(\gamma_{2})\leq\tau_{j}-\nu_{i}} \left(
\sum_{\mathcal{H}_{\delta}(\gamma_{1})+\mathcal{H}_{E}(\gamma_{2}) = \tau
_{j}-\nu_{i}} c_{i,j,\gamma_{1},\gamma_{2}}x^{\gamma_{1}}\right)  \xi
^{\gamma_{2}}\\[0.1cm]
&  (\text{since $\tau_{j} = \mathcal{H}_{E}(e_{j})$ and $\nu_{i}\geq1$, so
that $\tau_{j}-\nu_{i}\leq m_{e_{j}}$})\\
&  \equiv\sum_{\mathcal{H}_{E}(\gamma)\leq m_{e_{j}}}c_{i,j,\gamma}%
(x)\,\xi^{\gamma},
\end{split}
\end{equation}
and this proves \eqref{eq:ralfabetaLemmaTecnico} in this case.\vspace{0.1cm}

We now assume that \eqref{eq:ToProveInductionRstarOperatoresxi} holds for some
$k\geq1$ (and for every multi-in\-dex of length $k$), and we prove that it
holds also \emph{for every given multi-index}
\[
I=(i_{1},\ldots,i_{k},i_{k+1})
\]
of length $k+1$. To this end we first observe that, setting $J=(i_{1}%
,\ldots,i_{k}),$ by using the inductive hypothesis, jointly with Theorem
\ref{thm:Adrift}, we can write
\[%
\begin{split}
&  \widetilde{X}_{i_{k+1}}\cdots\widetilde{X}_{i_{1}}-X_{i_{k+1}}\cdots
X_{i_{1}}\\
&  \qquad=\widetilde{X}_{i_{k+1}}\big(\widetilde{X}_{i_{k}}\cdots
\widetilde{X}_{i_{1}}\big)-X_{i_{k+1}}X_{i_{k}}\cdots X_{i_{1}}\\
&  \qquad=(X_{i_{k+1}}+R_{i_{k+1}})\Big(X_{i_{k}}\cdots X_{i_{1}}%
+\sum\nolimits_{\beta\neq0}r_{\alpha,\beta}(x,\xi) \partial_{x}^{\alpha
}\partial_{\xi}^{\beta}\Big)\\
&  \qquad\qquad-X_{i_{k+1}}X_{i_{k}}\cdots X_{i_{1}}\\
&  \qquad=X_{i_{k+1}}\Big(\sum\nolimits_{\beta\neq0}r_{\alpha,\beta}%
(x,\xi)\partial_{x}^{\alpha}\partial_{\xi}^{\beta}\Big)+R_{i_{k+1}}X_{i_{k}%
}\cdots X_{i_{1}}\\
&  \qquad\qquad+R_{i_{k+1}}\Big(\sum\nolimits_{\beta\neq0}r_{\alpha,\beta
}(x,\xi)\partial_{x}^{\alpha}\partial_{\xi}^{\beta}\Big)
\end{split}
\]
(where each $r_{\alpha,\beta}$ is a polynomial function of the form
\eqref{eq:ralfabetaLemmaTecnico}, possibly depending on $J$, and not all of
them vanish). As a consequence, if we define
\begin{align*}
\bullet)  &  \,\,\mathcal{A}_{1}=X_{i_{k+1}}\Big(\sum\nolimits_{\beta\neq
0}r_{\alpha,\beta}(x,\xi) \partial_{x}^{\alpha} \partial_{\xi}^{\beta}\Big);\\
\bullet)  &  \,\,\mathcal{A}_{2}=R_{i_{k+1}}X_{i_{k}}\cdots X_{i_{1}};\\
\bullet)  &  \,\,\mathcal{A}_{3} = R_{i_{k+1}}\Big(\sum\nolimits_{\beta\neq
0}r_{\alpha,\beta}(x,\xi) \partial_{x}^{\alpha}\partial_{\xi}^{\beta}\Big);
\end{align*}
we see that \eqref{eq:ToProveInductionRstarOperatoresxi} holds (for this given
multi-index $I$ of length $k+1$) if it holds \emph{for the three operators
$\mathcal{A}_{1}, \mathcal{A}_{2}$ and $\mathcal{A}_{3}$}. \medskip

\noindent-\,\,\emph{Analysis of $\mathcal{A}_{1}$.} We begin by noting that,
if we write
\[
X_{i_{k+1}} = \sum_{j = 1}^{n} a_{i_{k+1},j}(x)\partial_{x_{j}}%
\]
(for suitable polynomial functions $a_{i_{k+1},j}$, $\delta_{\lambda}%
$-homogeneous of degree $\sigma_{j}-\nu_{i_{k+1}}$, see again Remark
\ref{rem:ConseguenzeH1H2}), by definition of $\mathcal{A}_{1}$ we have
\begin{align*}
&  \mathcal{A}_{1} = \sum_{\beta\neq0}\Big(\sum_{j = 1}^{n}a_{i_{k+1}%
,j}(x)\cdot(\partial_{x_{j}}r_{\alpha,\beta})(x,\xi)\Big)\partial^{\alpha}_{x}
\partial^{\beta}_{\xi}\\
&  \qquad\qquad+ \sum_{\beta\neq0}\sum_{j = 1}^{n} a_{i_{k+1},j}(x)\cdot
r_{\alpha,\beta}(x,\xi)\,\partial^{\alpha+e_{j}}_{x} \partial^{\beta}_{\xi}\\
&  \qquad(\text{setting $\alpha^{\prime}= \alpha+e_{j}$})\\
&  \qquad\equiv\sum_{\beta\neq0}\mathfrak{a}_{\alpha,\beta}(x,\xi)\,D^{\alpha
}_{x} D^{\beta}_{\xi}+ \sum_{\beta\neq0}\mathfrak{b}_{\alpha^{\prime},\beta
}(x,\xi)\,\partial^{\alpha^{\prime}}_{x} \partial^{\beta}_{\xi},
\end{align*}
where $\mathfrak{a}_{\alpha,\beta},\,\mathfrak{b}_{\alpha^{\prime},\beta}$ are
polynomial functions, not all identically zero (notice that the coefficients
$a_{i_{k+1},j}$ are not all identically zero since $X_{i_{k+1}}$ is a
non-va\-ni\-shing vector field, while the functions $r_{\alpha,\beta}$ are not
all identically zero by the inductive hypothesis). Hence, $(\bigstar)$ holds
for $\mathcal{A}_{1}$ \medskip

As for the \emph{explicit form} of $\mathfrak{a}_{\alpha,\beta}$ and
$\mathfrak{b}_{\alpha^{\prime},\beta}$ it suffices to note that, since
$r_{\alpha,\beta}$ is of the form \eqref{eq:ralfabetaLemmaTecnico} (by the
inductive hypothesis), we can write
\begin{align*}
\mathfrak{a}_{\alpha,\beta}(x,\xi)  &  = \sum_{j = 1}^{n} a_{i_{k+1}%
,j}(x)\cdot\Big[\partial_{x_{j}}\Big(\sum_{\mathcal{H}_{E}(\gamma)\leq
m_{\beta}}c_{\alpha,\beta,\gamma}(x)\,\xi^{\gamma}\Big)\Big]\\
&  = \sum_{\mathcal{H}_{E}(\gamma)\leq m_{\beta}}\Big(\sum_{j = 1}^{n}
a_{i_{k+1},j}(x)\cdot(\partial_{x_{j}}c_{\alpha,\beta,\gamma})(x)\Big)\xi
^{\gamma}\\
&  \equiv\sum_{\mathcal{H}_{E}(\gamma)\leq m_{\beta}}\mathfrak{p}%
_{\alpha,\beta,\gamma}(x)\,\xi^{\gamma};\\[0.1cm]
\mathfrak{b}_{\alpha^{\prime},\beta}(x,\xi)  &  = a_{i_{k+1},j}(x)\cdot
r_{\alpha,\beta}(x,\xi)\\
&  = \sum_{\mathcal{H}_{E}(\gamma)\leq m_{\beta}} \big(a_{i_{k+1},j}(x)\cdot
c_{\alpha,\beta,\gamma}(x)\big)\xi^{\gamma}\\
&  \equiv\sum_{\mathcal{H}_{E}(\gamma)\leq m_{\beta}}{\mathfrak{q}}%
_{\alpha^{\prime},\beta,\gamma}(x)\,\xi^{\gamma}%
\end{align*}
(for suitable polynomial functions $\mathfrak{p}_{\alpha,\beta,\gamma
},\,\mathfrak{q}_{\alpha^{\prime},\beta,\gamma}$ only depending on $x$), and
this fully proves the validity of \eqref{eq:ToProveInductionRstarOperatoresxi}
for the operator $\mathcal{A}_{1}$. \medskip

\noindent-\,\,\emph{Analysis of $\mathcal{A}_{2}$.} We first observe that,
since each $X_{i}$ has polynomial coeffi\-cient functions (see Remark
\ref{rem:ConseguenzeH1H2}), we can write
\[
X_{i_{k}}\cdots X_{i_{1}} = \sum_{\alpha} p_{\alpha}(x)\,\partial^{\alpha}%
_{x},
\]
for suitable polynomials $p_{\alpha}$, not all identically zero. From this, by
the explicit expression of $R_{i_{k+1}}$ (see \eqref{eq:explicitRijStep1} with
$i_{k+1}$ in place of $i$), we obtain
\begin{align*}
R_{i_{k+1}}X_{i_{k}}\cdots X_{i_{1}}  &  = \sum_{j = 1}^{p} r_{i_{k+1}%
,j}(x,\xi)\,\partial_{\xi_{j}} \left(  \sum_{\alpha} p_{\alpha} (x)\,\partial
^{\alpha}_{x}\right) \\
&  = \sum_{\alpha}\sum_{j = 1}^{p} \big(r_{i_{k+1},j}(x,\xi)\cdot p_{\alpha
}(x)\big)\partial^{\alpha}_{x}\partial_{\xi_{j}}\\
&  \equiv\sum_{\mathcal{H}(\beta) = 1}\mathfrak{c}_{\alpha,\beta}%
(x,\xi)\,\partial^{\alpha}_{x} \partial^{\beta}_{\xi},
\end{align*}
where each $\mathfrak{c}_{\alpha,\beta}$ (with $\beta= e_{j}$ for some $1\leq
j\leq p$) is a non-zero polyno\-mial function (note that the coefficients
$r_{i_{k+1},j}$ are not all identically zero, since $R_{i_{k+1}}$ is a
non-vanishing vector field). Hence, $(\bigstar)$ holds also for $\mathcal{A}%
_{2}$

As for the \emph{explicit form} of $\mathfrak{c}_{\alpha,\beta}$ it suffices
to observe that, if $\beta= e_{j}$ for some index $1\leq j\leq p$, by
exploiting \eqref{eq:formrijStep1} (with $i = i_{k+1}$) we can write
\begin{align*}
\mathfrak{c}_{\alpha,\beta}(x,\xi)  &  = r_{i_{k+1},j}(x,\xi)\cdot p_{\alpha
}(x) = \sum_{\mathcal{H}_{E}(\gamma)\leq m_{e_{j}}}\big(c_{i_{k+1},j,\gamma
}(x)\cdot p_{\alpha}(x)\big)\xi^{\gamma},
\end{align*}
and this fully establishes \eqref{eq:ToProveInductionRstarOperatoresxi} for
$\mathcal{A}_{2}$ (since $\beta= e_{j}$). \medskip

\noindent-\,\,\emph{Analysis of $\mathcal{A}_{3}$}. First of all we observe
that, by using once again the explicit expression of $R_{i_{k+1}}$ (see
\eqref{eq:explicitRijStep1} with $i_{k+1}$ in place of $i$), we have
\begin{align*}
&  \mathcal{A}_{3} = \sum_{\beta\neq0}\Big(\sum_{j = 1}^{p}r_{i_{k+1},j}%
(x,\xi)\cdot(\partial_{\xi_{j}}r_{\alpha,\beta})(x,\xi)\Big)\partial^{\alpha
}_{x} \partial^{\beta}_{\xi}\\
&  \qquad\qquad+ \sum_{\beta\neq0}\sum_{j = 1}^{p} r_{i_{k+1},j}(x,\xi)\cdot
r_{\alpha,\beta}(x,\xi)\,\partial^{\alpha}_{x} \partial^{\beta+e_{j}}_{\xi}\\
&  \qquad(\text{setting $\beta^{\prime}= \beta+e_{j}$, so that $\mathcal{H}%
(\beta^{\prime})\geq2$})\\
&  \qquad\equiv\sum_{\beta\neq0}\mathfrak{f}_{\alpha,\beta}(x,\xi
)\,\partial^{\alpha}_{x} \partial^{\beta}_{\xi}+ \sum_{\mathcal{H}%
(\beta^{\prime})\geq2}\mathfrak{h}_{\alpha,\beta^{\prime}}(x,\xi
)\,\partial^{\alpha}_{x} \partial^{\beta^{\prime}}_{\xi},
\end{align*}
where $\mathfrak{f}_{\alpha,\beta},\,\mathfrak{h}_{\alpha,\beta^{\prime}}$ are
polynomial functions, not all identically zero (notice that while the
coefficients $r_{i_{k+1},j}$ are not all identically zero since $R_{i_{k+1}}$
is a non-va\-ni\-shing vector field, while the functions $r_{\alpha,\beta}$
are not all identically zero by the inductive hypothesis). Hence, $(\bigstar)$
holds for $\mathcal{A}_{3}$. \medskip

As for the \emph{explicit form} of $\mathfrak{f}_{\alpha,\beta}$ and
$\mathfrak{h}_{\alpha,\beta^{\prime}}$, we begin by noting that, since
$r_{\alpha,\beta}$ is of the form \eqref{eq:ralfabetaLemmaTecnico} (by the
inductive hypothesis), again by using \eqref{eq:formrijStep1} (with the choice
$i = i_{k+1}$) we can write
\begin{align*}
&  \mathfrak{f}_{\alpha,\beta}(x,\xi)\\
&  \quad= \sum_{j = 1}^{p} \Big(\sum_{\mathcal{H}_{E}(\gamma_{1})\leq
m_{e_{j}}} c_{i_{k+1},j,\gamma_{1}}(x)\,\xi^{\gamma_{1}}\Big)\cdot
\Big[\partial_{\xi_{j}}\Big(\sum_{\mathcal{H}_{E}(\gamma_{2})\leq m_{\beta}%
}c_{\alpha,\beta,\gamma_{2}}(x)\,\xi^{\gamma_{2}}\Big)\Big]\\
&  \quad= \sum_{j = 1}^{p}\sum
_{\begin{subarray}{c} \mathcal{H}_{E}(\gamma_1)\leq m_{e_j} \\ \mathcal{H}_E(\gamma_2)\leq m_\beta \end{subarray} }%
\Big( (\gamma_{2})_{j}\cdot c_{i_{k+1},j,\gamma_{1}}(x)\cdot c_{\alpha
,\beta,\gamma_{2}}(x)\Big)\xi^{\gamma_{1}+\gamma_{2}-e_{j}};
\end{align*}
on the other hand, given any $1\leq j\leq p$ and any pair of multi-indexes
$\gamma_{1},\gamma_{2}$ satisfying $\mathcal{H}_{E}(\gamma_{1})\leq m_{e_{j}%
},\,\mathcal{H}_{E}(\gamma_{2})\leq m_{\beta}$, we have

\begin{itemize}
\item[i)] the function $x\mapsto(\gamma_{2})_{j}\cdot c_{i_{k+1},j,\gamma_{1}%
}(x)\cdot c_{\alpha,\beta,\gamma_{2}}(x)$ is a polynomial function only
depending on $x$;

\item[ii)] $\mathcal{H}_{E}(\gamma_{1}+\gamma_{2}-e_{j}) \leq m_{e_{j}%
}+m_{\beta}-\mathcal{H}_{E}(e_{j}) = m_{\beta}-1\leq m_{\beta}$
\end{itemize}

(where we have used the fact that $\mathcal{H}_{E}(\gamma_{1})\leq m_{e_{j}%
},\,\mathcal{H}_{E}(\gamma_{2})\leq m_{\beta}$, jointly with the defini\-tion
\eqref{eq:defHbeta}), and this proves that $\mathfrak{f}_{\alpha,\beta}$ is of
the form \eqref{eq:ralfabetaLemmaTecnico}. \vspace{0.1cm}

Similarly, given any multi-index $\beta^{\prime}$ with $\beta^{\prime}=
\beta+e_{j}$ (for some $1\leq j\leq p$), by exploiting \eqref{eq:formrijStep1}
and the inductive hypothesis we can write
\begin{align*}
&  \mathfrak{h}_{\alpha,\beta^{\prime}}(x,\xi)\\
&  \quad= \sum_{j = 1}^{p} \Big(\sum_{\mathcal{H}_{E}(\gamma_{1})\leq
m_{e_{j}}} c_{i_{k+1},j,\gamma_{1}}(x)\,\xi^{\gamma_{1}}\Big)\cdot
\Big(\sum_{\mathcal{H}_{E}(\gamma_{2})\leq m_{\beta}}c_{\alpha,\beta
,\gamma_{2}}(x)\,\xi^{\gamma_{2}}\Big)\\
&  \quad= \sum_{j = 1}^{p}\sum
_{\begin{subarray}{c} \mathcal{H}_{E}(\gamma_1)\leq m_{e_j} \\ \mathcal{H}_E(\gamma_2)\leq m_\beta \end{subarray} }%
\Big(c_{i_{k+1},j,\gamma_{1}}(x)\cdot c_{\alpha,\beta,\gamma_{2}}%
(x)\Big)\xi^{\gamma_{1}+\gamma_{2}};
\end{align*}
on the other hand, given any $1\leq j\leq p$ and any pair of multi-indexes
$\gamma_{1},\gamma_{2}$ satisfying $\mathcal{H}_{E}(\gamma_{1})\leq m_{e_{j}%
},\,\mathcal{H}_{E}(\gamma_{2})\leq m_{\beta}$, we see that

\begin{itemize}
\item[i)] the function $x\mapsto c_{i_{k+1},j,\gamma_{1}}(x)\cdot
c_{\alpha,\beta,\gamma_{2}}(x)$ is a polynomial function, which only depends
on $x$;

\item[ii)] $\mathcal{H}_{E}(\gamma_{1}+\gamma_{2}) \leq m_{e_{j}}+m_{\beta} =
\mathcal{H}_{E}(\beta+e_{j}) -2 = \mathcal{H}_{E}(\beta^{\prime})-2 \leq
m_{\beta^{\prime}}$
\end{itemize}

(where we have used the fact that $\mathcal{H}_{E}(\gamma_{1})\leq m_{e_{j}%
},\,\mathcal{H}_{E}(\gamma_{2})\leq m_{\beta}$, jointly with the relation
$\beta^{\prime}= \beta+e_{j}$ and the definition \eqref{eq:defHbeta}), and
this proves that also the polynomial $\mathfrak{h}_{\alpha,\beta^{\prime}}$ is
of the form \eqref{eq:ralfabetaLemmaTecnico}. \vspace{0.1cm}

Gathering these facts, we conclude that the claimed
\eqref{eq:ToProveInductionRstarOperatoresxi} also holds for the operator
$\mathcal{A}_{3}$, and the proof of the lemma is finally complete.
\end{proof}

\medskip

\begin{proof}
[Proof of Proposition \ref{prop:S1S2vere}]The validity of property (S1) is
\emph{precisely} the first part in the statement of Lemma
\ref{lem:AggiuntoOperaInxi} (see \eqref{eq:formRstarPerS1}). As for the
validity of (S2), instead, it can be proved by \emph{following exactly} the
same lines as the proof of \cite[Thm.\,4.3]{BB-lift}; we present it here with
all the details for the sake of completeness. \vspace{0.1cm}

To begin with, we consider the continuous homogeneous norm $\varrho_{E}$
introduced in \eqref{eq:varrhod} and associated with the family $\{E_{\lambda
}\}_{\lambda}$, that is%
\[
\varrho_{E}(\xi)=\sum_{j=1}^{p}|\xi_{j}|^{1/\tau_{j}}%
\]
(where the $\tau_{j}$'s are the exponents of $E_{\lambda}$, see
\eqref{eq:Dlambdalift}); accordingly, we choose a smooth function $\Theta\in
C_{0}^{\infty}(\mathbb{R}^{p})$ such that \medskip

a)\,\,$0\leq\Theta\leq1$ on $\mathbb{R}^{p}$; \vspace{0.05cm}

b)\thinspace\thinspace setting $\mathcal{O}_{r}=\{\xi\in\mathbb{R}%
^{p}:\,\varrho_{E}(\xi)<r\}$ (with $r>0$), we have
\[
\text{$\Theta\equiv1$ on $\mathcal{O}_{1}$ \quad and \quad$\Theta\equiv0$ on
$\mathbb{R}^{p}\setminus\mathcal{O}_{2}$.}%
\]
We then define a sequence $\{\theta_{j}\}_{j}\subseteq C_{0}^{\infty
}(\mathbb{R}^{p})$ by setting
\[
\theta_{j}(\xi)=\Theta\big(E_{2^{-j}}(\xi)\big)\qquad(\text{with $\xi
\in\mathbb{R}^{p},\,j\in\mathbb{N}$})
\]
and we turn to prove that this sequence satisfies i)-iii) in property (S2).
\vspace{0.1cm}

As for i) and ii) we first notice that, by definition of $\theta_{j}$ (and
owing to property a) of $\Theta$), we clearly have $0\leq\theta_{j}\leq1$ on
$\mathbb{R}^{p}$; moreover, using the $E_{\lambda}$-homo\-geneity of the
function $\varrho_{E}$, together with property b) of $\Theta$, we also have
\[
\mathcal{O}_{2^{j}}\Subset\Omega_{j}=\{\xi\in\mathbb{R}^{p}:\,\theta_{j}%
(\xi)=1\}\subseteq\mathcal{O}_{2^{j+1}}.
\]
From this, we immediately deduce that
\[
\Omega_{j}\Subset\Omega_{j+1}\quad\text{and}\quad\bigcup_{j\geq1}\Omega
_{j}=\mathbb{R}^{p}.
\]
Hence, we are left to prove the validity of iii). To this end, we arbitrarily
fix a compact set $K\subseteq\mathbb{R}^{n}$ and, according to
\eqref{eq:RadjointEsplicitar}, we write
\[
\mathcal{R}^{\ast}=\sum_{\beta\neq0}r_{\alpha,\beta}(x,\xi)\partial
_{x}^{\alpha}\partial_{\xi}^{\beta}%
\]
(for suitable smooth functions $r_{\alpha,\beta}$, not all identically zero).

Then, exploiting Lemma \ref{lem:AggiuntoOperaInxi}, for every $\alpha,\beta$
we have the following computation, holding true for every $j\in\mathbb{N}$ and
every $x\in K,\,\xi\in\mathbb{R}^{p}$ (here we use the notation $\mathbf{1}%
_{A}$ for the indicator function of a set $A$):
\begin{align*}
&  |r_{\alpha,\beta}(x,\xi)\partial_{\xi}^{\beta}\theta_{j}(\xi)|\leq
\sum_{\mathcal{H}_{E}(\gamma)\leq m_{\beta}}|c_{\alpha,\beta,\gamma}%
(x)|\cdot|\xi^{\gamma}|\cdot|\partial_{\xi}^{\beta}\theta_{j}(\xi)|\\
&  \qquad(\text{since $\theta_{j}$ is constant out of $C_{j}=\{2^{j}%
\leq\varrho_{E}(\xi)\leq2^{j+1}\}$})\\
&  \qquad\leq\sum_{\mathcal{H}_{E}(\gamma)\leq m_{\beta}}\max_{K}%
|c_{\alpha,\beta,\gamma}|\cdot|\xi^{\gamma}|\cdot|\partial_{\xi}^{\beta}%
\theta_{j}(\xi)|\cdot\mathbf{1}_{C_{j}}(\xi)\\
&  \qquad\leq\sum_{\mathcal{H}_{E}(\gamma)\leq m_{\beta}}\max_{K}%
|c_{\alpha,\beta,\gamma}|\cdot|\xi^{\gamma}|\cdot\big[\max_{\mathbb{R}^{p}%
}|\partial^{\beta}\Theta|\cdot(2^{-j})^{\mathcal{H}_{E}(\beta)}\big]\cdot
\mathbf{1}_{C_{j}}(\xi)\\
&  \qquad\leq c_{\alpha,\beta}^{\prime}(K)\sum_{\mathcal{H}_{E}(\gamma)\leq
m_{\beta}}|\xi^{\gamma}|\cdot(2^{-j})^{\mathcal{H}_{E}(\beta)}\cdot
\mathbf{1}_{C_{j}}\\
&  \qquad(\text{writing $\xi=(E_{2^{j}}\circ E_{2^{-j}})(\xi)$})\\
&  \qquad\leq c_{\alpha,\beta}^{\prime}(K)\sum_{\mathcal{H}_{E}(\gamma)\leq
m_{\beta}}(2^{-j})^{\mathcal{H}_{E}(\beta)-\mathcal{H}_{E}(\gamma)}%
\cdot|E_{2^{-j}}(\xi)^{\gamma}|\cdot\mathbf{1}_{C_{j}}.
\end{align*}
We now observe that, since $\mathcal{H}_{E}(\gamma)\leq m_{\beta}%
=\mathcal{H}_{E}(\beta)-1$, we have
\[
(2^{-j})^{\mathcal{H}_{E}(\beta)-\mathcal{H}_{E}(\gamma)}\leq2^{-j};
\]
moreover, since the point $E_{2^{-j}}(\xi)$ belongs to the compact set
$C_{0}=\{1\leq\varrho_{E}\leq2\}$ when $\xi\in C_{j}$ (as $\varrho_{E}$ is
$E_{\lambda}$-homogeneous of degree $1$), we also have
\[
|E_{2^{-j}}(\xi)^{\gamma}|\cdot\mathbf{1}_{C_{j}}(\xi)\leq\max_{\mathcal{H}%
_{E}(\gamma)\leq m_{\beta}}\max_{\eta\in C_{0}}|\eta^{\gamma}|=c_{\beta
}^{\prime\prime}%
\]
for every $\xi\in\mathbb{R}^{p}$ and every multi-index $\gamma$ with
$\mathcal{H}_{E}(\gamma)\leq m_{\beta}$. \medskip

Gathering all these facts, we then conclude that
\begin{align*}
&  |r_{\alpha,\beta}(x,\xi)\partial_{\xi}^{\beta}\theta_{j}(\xi)| \leq
c^{\prime}_{\alpha,\beta}(K) \sum_{\mathcal{H}_{E}(\gamma)\leq m_{\beta}}
c^{\prime\prime}_{\beta}\cdot2^{-j}\\
&  \qquad\leq c^{\prime}_{\alpha,\beta}(K)\cdot c^{\prime\prime}_{\beta}%
\cdot\mathrm{card}\{\gamma:\,\mathcal{H}_{E}(\gamma)\leq m_{\beta}\} \equiv
c_{\alpha,\beta}(K),
\end{align*}
for every $j\in\mathbb{N}$ and every $x\in K,\,\xi\in\mathbb{R}^{p}$, and this
proves that also assertion iii) in property (S2) is fulfilled (as
$c_{\alpha,\beta}(K)$ only depends on $\alpha,\beta$ and $K$).

This completes the proof.
\end{proof}

\bigskip

\begin{proof}
[Proof of Proposition \ref{prop:P1P2}]\noindent-\thinspace\thinspace
\emph{Proof of} (P1). Let $x\neq y\in\mathbb{R}^{n}$ be fixed. Owing to the
\emph{global} estimates of $\widetilde{\Gamma}$ in Theorem
\ref{thm:PropertiesGammGroups}\thinspace-\thinspace(1), to verify property
(P1) it suffices to show that
\begin{equation}
\xi\mapsto\varrho_{D}\big((y,0)^{-1}\star(x,\xi)\big)^{\nu-Q}\in
L^{1}(\mathbb{R}^{p}), \label{eq:ToProveP1}%
\end{equation}
where $\varrho_{D}$ is the continuous homogeneous norm on $\mathbb{G}$
introduced in \eqref{eq:varrhod} and associated with the family $\{D_{\lambda
}\}_{\lambda}$ in \eqref{eq:Dlambdalift}, that is,
\begin{equation}
\varrho_{D}(z,\zeta)=\sum_{i=1}^{n}|z_{i}|^{1/\sigma_{i}}+\sum_{i=1}^{p}%
|\zeta_{i}|^{1/\tau_{i}}=\varrho_{\delta}(z)+\varrho_{E}(\zeta).
\label{eq:varrhoDsomma}%
\end{equation}
while $Q>q$ is the homogeneous dimension of $\mathbb{G}$ (recall that, in the
statement of Theorem \ref{thm:PropertiesGammGroups}, $\Vert\cdot\Vert$ denotes
\emph{any} fixed homogeneous norm on $\mathbb{G}$).

Hence, we turn to establish the above \eqref{eq:ToProveP1}. To this end we
first observe that, by exploiting the change of variable $\xi= \Psi_{x,y}%
^{-1}(\zeta)$ introduced in Remark \ref{rem:CambiDiVariabile}, and using that
$\nu< q < Q$ (see \eqref{eq:nulessq}), we then get
\begin{equation}
\label{eq:ContoDaRiusareP2}%
\begin{split}
\int_{\mathbb{R}^{p}}  &  \varrho_{D}^{\nu-Q}\big((y,0)^{-1}\star
(x,\xi))\big)\,d\xi= \int_{\mathbb{R}^{p}}\varrho_{D}^{\nu-Q}\big((y,0)^{-1}%
\star(x,\Psi_{x,y}^{-1}(\zeta))\big)\,d\zeta\\
&  \qquad= \int_{\{\varrho_{E}\leq1\}} \varrho_{D}^{\nu-Q}\big((y,0)^{-1}%
\star(x,\Psi_{x,y}^{-1}(\zeta))\big)\,d\zeta\\
&  \qquad\qquad\quad+ \int_{\{\varrho_{E} > 1\}}\varrho_{D}^{\nu
-Q}\big((y,0)^{-1}\star(x,\Psi_{x,y}^{-1}(\zeta))\big)\,d\zeta\\
&  \qquad\equiv\mathrm{I}_{1}+\mathrm{I}_{2}.
\end{split}
\end{equation}
Now, the integral $\mathrm{I}_{1}$ is finite, since it is the integral of a
\emph{continuous function} on the \emph{compact set} $\{\varrho_{E}\leq1\}$:
indeed, since $x\neq y$ are fixed, we have
\[
(x,\Psi_{x,y}^{-1}(\zeta))\neq(y,0)\,\,\Longleftrightarrow\,\, (y,0)^{-1}%
\star(x,\Psi_{x,y}^{-1}(\zeta))\neq0\quad\text{for all $\zeta\in\mathbb{R}%
^{p}$},
\]
and thus the map
\[
\zeta\mapsto\varrho_{D}^{\nu-Q}((y,0)^{-1}\star(x,\Psi_{x,y}^{-1}(\zeta)))
\]
is continuous on $\mathbb{R}^{p}$ (as the same is true of $\varrho_{D}$).

As for the integral $\mathrm{I}_{2}$, instead, using the decomposition
\eqref{eq:varrhoDsomma} (and recalling the very definition of $\Psi_{x,y}$ in
Remark \ref{rem:CambiDiVariabile}), we get
\begin{equation}
\label{eq:StimaChiaveP1}%
\begin{split}
0\leq\mathrm{I}_{2}  &  \leq\int_{\{\varrho_{E} > 1\}} \varrho_{E}^{\nu
-Q}\big( \pi_{p}\big((y,0)^{-1}\star(x,\Psi_{x,y}^{-1}(\zeta
))\big)\big)\,d\zeta\\
&  = \int_{\{\varrho_{E} > 1\}} \varrho_{E}^{\nu-Q}\big( \Psi_{x,y}(\Psi
_{x,y}^{-1}(\zeta))\big)\,d\zeta\\
&  = \int_{\{\varrho_{E} > 1\}} \varrho_{E}^{\nu-Q}(\zeta)\,d\zeta
\end{split}
\end{equation}
(where $\pi_{p}:\mathbb{R}^{N}\to\mathbb{R}^{p}$ denotes, as usual, the
projection of $\mathbb{R}^{N}$ onto $\mathbb{R}^{p}$); further\-more, since we
are assuming that \eqref{eq:nulessq} holds, it is easy to see that the
integral in the far right-hand side of \eqref{eq:StimaChiaveP1} is finite.
\vspace{0.05cm}

Indeed, setting $\mathcal{C}_{j}=\big\{\zeta\in\mathbb{R}^{p}:\,2^{j-1}%
<\varrho_{E}(\zeta)\leq2^{j}\big\}$ (for every $j\in\mathbb{N}$), and
re\-calling \eqref{eq:Dlambdalift}, we have the following computation
\begin{equation}%
\begin{split}
&  \int_{\{\varrho_{E}>1\}}\varrho_{E}^{\nu-Q}(\zeta)\,d\zeta=\sum
_{j=1}^{+\infty}\int_{\mathcal{C}_{j}}\varrho_{E}^{\nu-Q}(\zeta)\,d\zeta\\
&  \qquad(\text{by the change of variable $\zeta=E_{2^{j}}(\upsilon)$})\\
&  \qquad=\sum_{j=1}^{+\infty}\Big(\int_{\{\upsilon:\,1/2<\varrho_{E}%
(\upsilon)\leq1\}}\varrho_{E}(\upsilon)^{\nu-Q}\,d\upsilon\Big)\cdot
(2^{j})^{\nu-Q+\mathcal{E}}\\
&  \qquad=\kappa_{0}\sum_{j=1}^{+\infty}(2^{\nu-Q+\mathcal{E}})^{j}%
=(\bigstar),
\end{split}
\label{eq:stimaIntermediaI2daUsare}%
\end{equation}
where $\mathcal{E}\in\mathbb{N}$ is the \emph{homogeneous dimension} of
$(\mathbb{R}^{p},E_{\lambda})$, namely $\mathcal{E}=\sum_{i=1}^{p}\tau_{i}.$
Moreover, using \eqref{eq:dimqGruppiOm} and \eqref{eq:nulessq}, we derive
\[
\nu-Q+\mathcal{E}=\nu-(q+\mathcal{E})+\mathcal{E}=\nu-q<0.
\]
Gathering these two facts, we then conclude that
\begin{equation}
(\bigstar)=c\Big(\frac{1}{1-2^{\nu-q}}-1\Big)<+\infty.
\label{eq:ConclusionI2DaUsare}%
\end{equation}

\noindent-\,\,\emph{Proof of} (P2). Let $y\in\mathbb{R}^{n}$ be fixed, and let
$K\subseteq\mathbb{R}^{n}$ be a compact set. Using once again the global
estimates of $\widetilde{\Gamma}_{0}$ in Theorem
\ref{thm:PropertiesGammGroups}\,-\,(1), and following the notation introduced
above, to verify property (P2) it suffices to show that
\begin{equation}
\label{eq:ToProvePropertyP2Equiv}(x,\xi)\mapsto\varrho_{D}^{\nu-Q}%
\big((y,0)^{-1}\star(x,\xi)\big)\in L^{1}(K\times\mathbb{R}^{p});
\end{equation}
hence, we turn to establish \eqref{eq:ToProvePropertyP2Equiv}. \vspace{0.05cm}

To begin with, by exploiting estimate \eqref{eq:ContoDaRiusareP2} we have
\begin{align*}
&  \int_{K\times\mathbb{R}^{p}}\varrho_{D}^{\nu-Q}\big((y,0)^{-1}\star
(x,\xi)\big)\,dx\,d\xi\\
&  \qquad= \int_{K} \Big(\int_{\mathbb{R}^{p}}\varrho_{D}^{\nu-Q}%
\big((y,0)^{-1}\star(x,\xi)\big)\,d\xi\Big)dx\\
&  \qquad= \int_{K}\Big(\int_{\{\varrho_{E}\leq1\}} \varrho_{D}^{\nu
-Q}\big((y,0)^{-1}\star(x,\Psi_{x,y}^{-1}(\zeta))\big)\,d\zeta\Big)dx\\
&  \qquad\qquad\quad+ \int_{K}\Big(\int_{\{\varrho_{E} > 1\}}\varrho_{D}%
^{\nu-Q}\big((y,0)^{-1}\star(x,\Psi_{x,y}^{-1}(\zeta))\big)\,d\zeta\Big)dx\\
&  \qquad= \int_{K\times\{\varrho_{E}\leq1\}}\varrho_{D}^{\nu-Q}%
\big((y,0)^{-1}\star(x,\Psi_{x,y}^{-1}(\zeta))\big)\,dx\,d\zeta\\
&  \qquad\qquad\quad+ \int_{K\times\{\varrho_{E}> 1\}}\varrho_{D}^{\nu
-Q}\big((y,0)^{-1}\star(x,\Psi_{x,y}^{-1}(\zeta))\big)\,dx\,d\zeta\\
&  \qquad\equiv\mathrm{J}_{1}+\mathrm{J}_{2}.
\end{align*}
Moreover, from the computations \emph{already performed} in the verification
of property (P1), it readily follows that the integral $\mathrm{J}_{2}$ is
finite. Indeed, by \eqref{eq:StimaChiaveP1}-to-\eqref{eq:ConclusionI2DaUsare},
we have
\begin{align*}
0\leq\mathrm{J}_{2}  &  = \int_{K}\Big(\int_{\varrho_{E} > 1} \varrho_{D}%
^{\nu-Q}((y,0)^{-1}\star(x,\Psi_{x,y}^{-1}(\zeta)))\,d\zeta\Big)dx\\
&  \leq\int_{K}\Big(\int_{\{\varrho_{E} > 1\}}\varrho_{E}^{\nu-Q}%
(\zeta)\,d\zeta\Big)dx\\
&  \leq\kappa_{0}\Big(\frac{1}{1-2^{\nu-q}}-1\Big)\cdot|K| < +\infty.
\end{align*}
Finally, we claim that the integral $\mathrm{J}_{1}$ is also finite. Indeed,
by using the \emph{measure-pre\-serving change of variable in $\mathbb{R}^{N}%
$} defined by
\[
(x,\zeta) = (x,\Psi_{x,y}(\xi))\equiv\mathcal{F}_{y}(x,\xi)
\]
(notice that $\mathcal{F}_{y}$ is a smooth diffeomorphism of $\mathbb{R}^{N}$
with Jacobian determi\-nant equal to $1$, due to the properties of $\Psi
_{x,y}$ in Remark \ref{rem:CambiDiVariabile}), we get
\begin{align*}
0\leq\mathrm{J}_{1}  &  = \int_{K^{\prime}} \varrho_{D}^{\nu-Q}\big((y,0)^{-1}%
\star(x,\xi)\big)\,dx\,d\xi= (\bigstar),
\end{align*}
where $K^{\prime}= \mathcal{F}_{y}^{-1}(K\times\{\varrho_{E}\leq1\})$ is a
\emph{compact subset of $\mathbb{R}^{N}$} (depending on the fixed $y$). From
this, by applying the change of variable
\[
(x,\xi) = (y,0)*(a,\alpha) \equiv\tau_{(y,0)}(a,\alpha),
\]
and since the Lebesgue measure is a Haar measure on $\mathbb{G}$, we obtain
\[
(\bigstar) = \int_{K^{\prime\prime}}\varrho_{D}^{\nu-Q}(a,\alpha
)\,da\,d\alpha,
\]
where $K^{\prime\prime}= \tau_{(y,0)}^{-1}(K^{\prime})$ is another
\emph{compact set in $\mathbb{R}^{N}$}.

Now, since $\varrho_{D}$ is a \emph{homogeneous norm} on the group
$\mathbb{G}=(\mathbb{R}^{N},\star,D_{\lambda})$ (with homogeneous dimension
$Q$), and since $\nu-Q>-Q$, by Remark \ref{rem:HomNormLloc} we have
\[
\varrho_{D}^{\nu-Q}\in L_{\mathrm{loc}}^{1}(\mathbb{R}^{N});
\]
thus, since $K^{\prime\prime}$ is compact, we conclude that $\mathrm{J}%
_{1}<+\infty$, as claimed.\medskip
\end{proof}

Gathering all the results established so far, we are finally ready to prove
our main result on the fundamental solution of $\mathcal{L}$.

\bigskip

\begin{proof}
[Proof of Theorem \ref{thm:ExistenceEstimLhGeneral}]Let $\mathbb{G}%
=(\mathbb{R}^{N},\star,D_{\lambda})$ and $\widetilde{X}=\{\widetilde{X}%
_{1},\ldots,\widetilde{X}_{m}\}$ be as in the sta\-tement of Theorem
\ref{thm:Adrift} (which applies, since the $X_{i}$'s satisfy assumption (H.1)
and since we assume that \eqref{eq:pgeqone} holds); we then consider the
operator
\[
\widetilde{\mathcal{L}}=\sum_{|I|=\nu}c_{I}\,\widetilde{X}_{I}.
\]
In view of Proposition \ref{prop:S1S2vere}, we know that the operator
$\widetilde{\mathcal{L}}$ satisfies proper\-ties (S1) and (S2) (that is,
$\widetilde{\mathcal{L}}$ is a \emph{saturable lifting of $\mathcal{L}$}). In
addition, by Proposition \ref{prop:P1P2}, the \emph{global fundamental
solution} $\widetilde{\Gamma}$ of $\widetilde{\mathcal{L}}$ (whose existence
is guaranteed by Theorem \ref{thm:PropertiesGammGroups}, since the
$\widetilde{X}_{i}$'s are left-invariant on the homogeneous group $\mathbb{G}%
$) satisfies the integrability properties (P1) and (P2). \vspace{0.1cm}

In view of these facts, we are then entitled to exploit \cite[Thm.\,2.1]%
{BB-lift} in the present context, thereby ensuring that the `saturated'
function
\begin{equation}
\Gamma(x,y)=\int_{\mathbb{R}^{p}}\widetilde{\Gamma}((x,\xi),(y,0))\,d\xi
=\int_{\mathbb{R}^{p}}\widetilde{\Gamma}_{0}((y,0)^{-1}\star(x,\xi
))\,d\xi\label{eq:defGammaLiftingRF}%
\end{equation}
(where $\widetilde{\Gamma}_{0}(x,\xi)=\widetilde{\Gamma}((x,\xi),(0,0))$ is
the global fundamental solution of $\widetilde{\mathcal{L}}$ with pole at the
origin, and $p=N-n\geq1$ - see \eqref{eq:pgeqone}), is indeed a \emph{global
fundamental solution} of $\mathcal{L}$, that is, for every $y\in\mathbb{R}%
^{n}$, we have
\[
\text{$\Gamma(\cdot,y)\in L_{\mathrm{loc}}^{1}(\mathbb{R}^{n})$\quad
and\quad$\Gamma(\cdot,y)$ fulfills property (b)}.
\]
The fact that $\Gamma(x,\cdot)\in L_{\mathrm{loc}}^{1}(\mathbb{R}^{n})$ for
every fixed $x\in\mathbb{R}^{n}$ will follow from (\ref{eq:GammaGammastarGC}),
that we will prove below. Let us now show that $\Gamma$ is also locally
integrable in the joint variables, thus completing the proof of point (a).
\vspace{0.05cm}

To this end, we proceed as in \cite{BB-lift} and we consider the following
steps. \vspace{0.1cm}

\noindent-\,\,\emph{Step 1}). In this first step we show that the map
\[
(x,y,\xi)\mapsto\widetilde{\Gamma}_{0}((y,0)^{-1}\star(x,\xi))
\]
is locally integrable on $\mathbb{R}^{n}\times\mathbb{R}^{n}\times
\mathbb{R}^{p}$. Let then $K_{1},K_{2}\subseteq\mathbb{R}^{n}$ and
$H\subseteq\mathbb{R}^{p}$ be compact sets. By Tonelli's Theorem, and by using
the change of variable
\[
(x,\xi) = (y,0)\star(z,\zeta),
\]
we can write
\begin{align*}
&  \int_{K_{1}\times K_{2}\times H}|\widetilde{\Gamma}_{0}((y,0)^{-1}%
\star(x,\xi))|\,dx\,dy\,d\xi\\
&  \qquad= \int_{K_{2}}\left(  \int_{\tau_{y}^{-1}(K_{1}\times H)}%
|\widetilde{\Gamma}_{0}(z,\zeta)|\,dz\,d\zeta\right)  d y = (\bigstar),
\end{align*}
where $\tau_{y}$ denotes the left-translation by $(y,0)$ on the group
$\mathbb{G}$. We now observe that, for every $y\in K_{2}$, the set $\tau
_{y}^{-1}(K_{1}\times H)$ is included in the \emph{compact set}
\[
\mathcal{K} = (K_{2}\times\{0\})^{-1}\star(K_{1}\times H);
\]
therefore, by recalling that $\widetilde{\Gamma}_{0}\in L^{1}_{\mathrm{loc}%
}(\mathbb{R}^{N})$, we obtain
\[
(\bigstar) \leq\int_{K_{2}}\left(  \int_{\mathcal{K}}|\widetilde{\Gamma}%
_{0}(z,\zeta)|\,dz\,d\zeta\right)  d y <+\infty.
\]
-\,\,\emph{Step 2}). In this second step we complete the demonstration of the
claimed \emph{joint integrability property} of $\Gamma$, that is, we prove
that $\Gamma\in L^{1}_{\mathrm{loc}}(\mathbb{R}^{n}\times\mathbb{R}^{n})$.

Let then $K_{1},K_{2}\subseteq\mathbb{R}^{n}$ be compact sets, and let
\[
\Phi: \mathbb{R}^{n}\times\mathbb{R}^{n}\times\mathbb{R}^{p}\to\mathbb{R}%
^{n}\times\mathbb{R}^{n}\times\mathbb{R}^{p} \quad\Phi(x,y,\xi) :=
\big(x,y,\Psi_{x,y}(\xi)\big).
\]
Taking into account the properties of $\Psi_{x,y}$ recalled in Remark
\ref{rem:CambiDiVariabile}, it is readily seen that $\Phi$ defines a smooth
diffeomorphism of $\mathbb{R}^{n}\times\mathbb{R}^{n}\times\mathbb{R}^{p}$,
with Jacobian determinant equal to $\pm1$. From this, by applying Fubini's
Theorem and by performing the change of variable $(x,y,\xi) = \Phi
^{-1}(u,v,\zeta)$, we get
\[%
\begin{split}
&  \int_{K_{1}\times K_{2}}|\Gamma(x,y)|\,d x\,d y \leq\int_{K_{1}\times
K_{2}\times\mathbb{R}^{p}}|\widetilde{\Gamma}_{0}((y,0)^{-1}\star
(x,\xi))|\,dx\,dy\,d\xi\\
&  \qquad= \int_{K_{1}\times K_{2}\times\mathbb{R}^{p}} |\widetilde{\Gamma
}_{0}\big((v,0)^{-1}\star(u,\Psi_{u,v}^{-1}(\zeta))\big)|\,du\,dv\,d\zeta\\
&  \qquad= \int_{K_{1}\times K_{2}\times\{\varrho_{E}<1\}}\{\ldots\}\,d u\,d
v\,d\zeta+ \int_{K_{1}\times K_{2}\times\{\varrho_{E}\geq1\}}\{\ldots\}\,d
u\,d v\,d\zeta\\
&  \qquad= \mathrm{I} + \mathrm{II},
\end{split}
\]
where, as usual, $\varrho_{E}$ is the continuous homogeneous norm on
$\mathbb{R}^{p}$ introduced in \eqref{eq:varrhod} and associated with the
family $\{E_{\lambda}\}_{\lambda}$ in \eqref{eq:Dlambdalift}. \vspace{0.1cm}

We now observe that, since the product $K_{1}\times K_{2}\times\{\varrho
_{E}<1\}$ is bounded in $\mathbb{R}^{n}\times\mathbb{R}^{n}\times
\mathbb{R}^{p}$, by Step 1) we easily infer that the integral $\mathrm{I}$ is finite.

As for integral $\mathrm{II}$ we notice that, by exploiting the global
estimates of $\widetilde{\Gamma}_{0}$ in Theorem
\ref{thm:PropertiesGammGroups}-(1) and by arguing exactly as in the proof of
Proposition \ref{prop:P1P2} (see, precisely,
\eqref{eq:StimaChiaveP1}-to-\eqref{eq:ConclusionI2DaUsare}), we obtain%
\[%
\begin{split}
\mathrm{II}  &  \leq c\, \int_{K_{1}\times K_{2}\times\{\varrho_{E}\geq
1\}}\varrho_{D}^{\nu-Q}\big((v,0)^{-1}\star(u,\Psi_{u,v}^{-1}(\zeta))\big) \,d
u\,d v\,d\zeta\\
&  (\text{since $\varrho_{D} = \varrho_{\delta}+\varrho_{E}$, see
\eqref{eq:varrhoDsomma}})\\
&  \leq c\, \int_{K_{1}\times K_{2}\times\{\varrho_{E}\geq1\}}\varrho_{E}%
^{\nu-Q} (\zeta) \,d u\,d v\,d\zeta\\
&  = c\cdot|K_{1}\times K_{2}|\,\int_{\{\varrho_{E}\geq1\}} \varrho_{E}%
^{\nu-Q}(\zeta)\,d\zeta< +\infty.
\end{split}
\]
This completes the proof of point (a). \medskip

Let us now prove point (c). Let $\varphi\in C_{0}^{\infty}(\mathbb{R}^{n})$,
and%
\[
u\left(  x\right)  =\int_{\mathbb{R}^{n}}\Gamma(x,y)\varphi(y)\,dy.
\]
Since $\Gamma\in L_{\mathrm{loc}}^{1}\left(  \mathbb{R}^{2n}\right)  $, then
$u\in L_{\mathrm{loc}}^{1}\left(  \mathbb{R}^{n}\right)  $. Moreover, for
every test function $\psi$,
\begin{align*}
\left\langle \mathcal{L}u,\psi\right\rangle  &  =\left\langle u,\mathcal{L}%
^{\ast}\psi\right\rangle =\int_{\mathbb{R}^{n}}u\left(  x\right)
\mathcal{L}^{\ast}\psi\left(  x\right)  dx\\
&  =\int_{\mathbb{R}^{n}}\left(  \int_{\mathbb{R}^{n}}\Gamma(x,y)\varphi
(y)\,dy\right)  \mathcal{L}^{\ast}\psi\left(  x\right)  dx\\
&  =\int_{\mathbb{R}^{n}}\varphi(y)\left(  \int_{\mathbb{R}^{n}}%
\Gamma(x,y)\mathcal{L}^{\ast}\psi\left(  x\right)  dx\right)  dy\\
&  (\text{by point (b)})\\
&  =-\int_{\mathbb{R}^{n}}\varphi(y)\psi\left(  y\right)  dy=\left\langle
-\varphi,\psi\right\rangle .
\end{align*}
Therefore $\mathcal{L}u=-\varphi$ in the distributional sense. Since
$\mathcal{L}$ is hypoelliptic, $u$ is (a.e. equal to) a smooth function
$u^{\ast}$ satisfying $\mathcal{L}u^{\ast}=-\varphi$ pointwise. \medskip

We now turn to prove the validity of properties \emph{(I)}\,-\,\emph{(V)} in
the statement of Theorem \ref{thm:ExistenceEstimLhGeneral}. \medskip

-\thinspace\thinspace\emph{Proof of (I)}. Our starting point is the following
general fact, whose demonstration is contained in the proof of
\cite[Lem.\,4.3]{BBB-fundsol}. \emph{Let
\[
\mathcal{V}=\big\{(x,y,\xi)\in\mathbb{R}^{n}\times\mathbb{R}^{n}%
\times\mathbb{R}^{p}:\,(x,\xi)\neq(y,0)\big\},
\]
and suppose $g\in C^{\infty}(\mathcal{V})$ is homogeneous of degree $d<q-Q$
with respect to the family of dilations $\{F_{\lambda}\}_{\lambda}$ in
$\mathbb{R}^{n}\times\mathbb{R}^{n}\times\mathbb{R}^{p}$ defined as follows
\[
F_{\lambda}(x,y,\xi)=(\delta_{\lambda}(x),\delta_{\lambda}(y),E_{\lambda}%
(\xi))
\]
\emph{(}recall that $Q>q$ is the homogeneous dimension of the homogeneous
group $\mathbb{G}$, and $E_{\lambda}$ is given by
\eqref{eq:Dlambdalift}\emph{)}. Moreover, let $Z$ be any smooth vector field
acting in the $(x,y)$-variables, homogeneous of degree $\mu>0$ with respect to
the family of dilations $\{G_{\lambda}\}_{\lambda}$ in $\mathbb{R}^{n}%
\times\mathbb{R}^{n}$ defined as follows
\[
G_{\lambda}(x,y)=(\delta_{\lambda}(x),\delta_{\lambda}(y)).
\]
Then, the following assertions hold:}

\begin{itemize}
\item[\emph{(i)}] \emph{for any fixed $(x,y)\in\mathbb{R}^{2n}\setminus
\mathbb{D}$, the map $\xi\mapsto g(x,y,\xi)\in L^{1}(\mathbb{R}^{p})$};

\item[\emph{(ii)}] \emph{$Z$ can pass under the integral sign as follows}
\[
Z\Big\{(x,y)\mapsto\int_{\mathbb{R}^{p}}g(x,y,\xi)\,d\xi\Big\}
= \int_{\mathbb{R}^{p}}Z\Big\{(x,y)\mapsto g(x,y,\xi)\Big\}\,d\xi
\quad\text{for $x\neq y$}.
\]

\end{itemize}

\noindent Now, since the function $g_{0}(x,y,\xi)=\widetilde{\Gamma}%
_{0}((y,0)^{-1}\star(x,\xi))\in C^{\infty}(\mathcal{V})$ and all of its
classical $(x,y)$\thinspace-\thinspace derivatives are $F_{\lambda}%
$-homogeneous of degree
\[
d\leq\nu-Q<q-Q
\]
(recall that $\widetilde{\Gamma}_{0}$ is $D_{\lambda}$-homogeneous of degree
$\nu-Q$), we can repeatedly apply this general fact to the
\emph{representation formula} \eqref{eq:defGammaLiftingRF}, thus obtaining
\[
\exists\,\,\partial_{x}^{\alpha}\partial_{y}^{\beta}\Gamma(x,y)=\int%
_{\mathbb{R}^{p}}\partial_{x}^{\alpha}\partial_{y}^{\beta}\Big\{(x,y)\mapsto
\widetilde{\Gamma}_{0}((y,0)^{-1}\star(x,\xi))\Big\}\,d\xi
\]
for every $(x,y)\in\mathbb{R}^{2n}\setminus\mathbb{D}$ and for every
multi-indexes $\alpha,\beta\in\mathbb{N}^{n}$. As a conseque\-nce, in order to
show that $\Gamma\in C^{\infty}(\mathbb{R}^{2n}\setminus\mathbb{D})$ it
suffices to prove that, if $g\in C^{\infty}(\mathcal{V})$ is a \emph{generic}
function as above, then the map
\begin{equation}
(x,y)\mapsto\Phi_{g}(x,y)=\int_{\mathbb{R}^{p}}g(x,y,\xi)\,d\xi
\label{eq:defPhigdausaredopo}%
\end{equation}
is \emph{continuous on $\mathbb{R}^{2n}\setminus\mathbb{D}$}. \vspace{0.1cm}

Let then $g:\mathcal{V}\rightarrow\mathbb{R}$ be as above, and let
$(x_{0},y_{0})\in\mathbb{R}^{2n}\setminus\mathbb{D}$ be fixed. Moreover, let
$r>0$ be so small that $\overline{B}_{r}(x_{0})\cap\overline{B}_{r}%
(y_{0})=\varnothing$, and let
\[
\{x_{k}\}_{k}\subseteq B_{r}(x_{0}),\qquad\{y_{k}\}_{k}\subseteq B_{r}(y_{0})
\]
be two sequences converging to $x_{0}$ and $y_{0}$, respectively, as
$k\rightarrow+\infty$. Since, by assumption, $g\in C^{\infty}(\mathcal{V})$,
for every fixed $\xi\in\mathbb{R}^{p}$ we have
\[
\lim_{k\rightarrow+\infty}g(x_{k},y_{k},\xi)=g(x_{0},y_{0},\xi);
\]
taking into account \eqref{eq:defPhigdausaredopo}, we then write
\[
\Phi_{g}(x_{k},y_{k})=\int_{\{\varrho_{E}\leq1\}}g(x_{k},y_{k},\xi)\,d\xi
+\int_{\{\varrho_{E}>1\}}g(x_{k},y_{k},\xi)\,d\xi
\]
(here, $\varrho_{E}$ is the continuous homogeneous norm introduced in
\eqref{eq:varrhod} and associa\-ted with the family $\{E_{\lambda}\}_{\lambda
}$ in \eqref{eq:Dlambdalift}), and we provide an integrable dominating
function $\Theta$ for the above integrals, independent of $k$. As for the
first integral we observe that, setting
\[
K:=\overline{B}_{r}(x_{0})\times\overline{B}_{r}(y_{0})\times\{\varrho_{E}%
\leq1\},
\]
by the choice of $r$ we see that $K$ is a compact subset of $\mathcal{V}$;
thus, since $g\in C^{\infty}(\mathcal{V})$, there exists a constant $C>0$ such
that
\[
|g(x_{k},y_{k},\xi)|\leq C\quad\text{for every $k\in\mathbb{N},\,\xi
\in\{\varrho_{E}\leq1\}$},
\]
and hence we can choose the \emph{constant function} $\Theta=C$ as an
integrable domina\-ting function in this case. As for the second integral,
instead, we need to exploit a more delicate argument, which is based on the
\emph{$F_{\lambda}$-homogeneity of $g$}.

To begin with, we fix $\rho_{0}>0$ so large that
\[
B_{r}(x_{0}),\,B_{r}(y_{0})\subseteq\{\varrho_{\delta}\leq\rho_{0}\}
\]
(here, $\varrho_{\delta}$ is the continuous homogeneous norm introduced in
\eqref{eq:varrhod} and associa\-ted with the family $\{\delta_{\lambda
}\}_{\lambda}$ in Assumption (H1)), and we set
\[
\mathcal{K}=\{\varrho_{\delta}\leq\rho_{0}\}\times\{\varrho_{\delta}\leq
\rho_{0}\}\times\{\varrho_{E}=1\}\subseteq\mathbb{R}^{n}\times\mathbb{R}%
^{n}\times\mathbb{R}^{p}.
\]
Since, obviously, $\mathcal{K}$ is a compact subset of $\mathcal{V}$, and
since $g\in C^{\infty}(\mathcal{V})$, it is possible to find a positive
constant $\kappa_{0}>0$ such that
\begin{equation}
|g(x_{k},y_{k},\xi)|\leq\kappa_{0}\quad\text{for every $k\in\mathbb{N}%
,\,\xi\in\{\varrho_{E}=1\}$.} \label{eq.boundgcompact}%
\end{equation}
On the other hand, if $\xi\in\mathbb{R}^{p}$ is such that $\varrho_{E}(\xi
)>1$, and if we set
\[
\lambda:=1/\varrho_{E}(\xi)\in(0,1),
\]
it is readily seen that $F_{\lambda}(x_{k},y_{k},\xi)\in\mathcal{K}$ for every
$k\in\mathbb{N}$; thus, by \eqref{eq.boundgcompact} and the $F_{\lambda}%
$\thinspace\--\thinspace homogeneity of $g$, we obtain the estimate
\begin{align*}
|g(x_{k},y_{k},\xi)|  &  =\big|g\big(F_{1/\lambda}(F_{\lambda}(x_{k},y_{k}%
,\xi))\big)\big|\\
&  =\lambda^{-d}|g(F_{\lambda}(x_{k},y_{k},\xi))|\\
&  \leq\kappa_{0}\,\lambda^{-d}=\kappa_{0}\,\varrho_{E}^{d}(\xi),
\end{align*}
holding true for every $k\in\mathbb{N}$ and every $\xi\in\{\varrho_{E}>1\}$.
In view of this estimate, we can then choose $\Theta=\kappa_{0}\varrho_{E}%
^{d}$ as a dominating function in this case, whose integrability can be proved
by arguing exactly as in Proposition \ref{prop:P1P2} (see, precisely,
\eqref{eq:stimaIntermediaI2daUsare} with $d$ in place of $\nu-Q$), taking into
account that
\[
d+\mathcal{E}=d+(Q-q)<0
\]
(see \eqref{eq:dimqGruppiOm} and recall that $\mathcal{E}=\sum_{j=1}^{p}%
\tau_{j}$). We may therefore exploit the Dominated Convergence Theorem to
conclude that%
\begin{align*}
\exists\,\,\lim_{k\rightarrow+\infty}\Phi_{g}(x_{k},y_{k})  &  =\int%
_{\{\varrho_{E}\leq1\}}g(x_{0},y_{0},\xi)\,d\xi+\int_{\{\varrho_{E}%
>1\}}g(x_{0},y_{0},\xi)\,d\xi\\
&  =\Phi_{g}(x_{0},y_{0}),
\end{align*}
which proves the continuity of $\Phi_{g}$ out of $\mathbb{D}$ (by the
arbitrariness of $x_{0},y_{0}$).

\medskip

-\thinspace\thinspace\emph{Proof of (II)}. Let $x\neq y\in\mathbb{R}^{n}$ be
fixed, and let $\lambda>0$. Recalling that $\widetilde{\Gamma}_{0}$ is
$D_{\lambda}$\--homo\-geneous of degree $\nu-Q$, by performing the change of
variable
\[
\xi=E_{\lambda}(\xi^{\prime})
\]
in the \emph{representation formula} \eqref{eq:defGammaLiftingRF} we get
\begin{align*}
&  \Gamma(\delta_{\lambda}(x),\delta_{\lambda}(y))=\int_{\mathbb{R}^{p}%
}\widetilde{\Gamma}_{0}((\delta_{\lambda}(y),0)^{-1}\star(\delta_{\lambda
}(x),\xi))\,d\xi\\
&  \qquad=\lambda^{\mathcal{E}}\int_{\mathbb{R}^{p}}\widetilde{\Gamma}%
_{0}((\delta_{\lambda}(y),0)^{-1}\star(\delta_{\lambda}(x),E_{\lambda}%
(\xi^{\prime})))\,d\xi^{\prime}\\
&  \qquad(\text{since $D_{\lambda}=(\delta_{\lambda},E_{\lambda})$, see
\eqref{eq:Dlambdalift}})\\
&  \qquad=\lambda^{\mathcal{E}}\int_{\mathbb{R}^{p}}\widetilde{\Gamma}%
_{0}\big(D_{\lambda}((y,0)^{-1}\star(x,\xi^{\prime}))\big)\,d\xi^{\prime}\\
&  \qquad=\lambda^{\mathcal{E}+\nu-Q}\int_{\mathbb{R}^{p}}\widetilde{\Gamma
}_{0}((y,0)^{-1}\star(x,\xi^{\prime}))\,d\xi^{\prime}\\
&  \qquad=\lambda^{\nu-(Q-\mathcal{E})}\,\Gamma(x,y),
\end{align*}
where $\mathcal{E}=\sum_{j=1}^{p}\tau_{j}$, see \eqref{eq:Dlambdalift}. This,
together with \eqref{eq:dimqGruppiOm}, gives \
\[
\Gamma(\delta_{\lambda}(x),\delta_{\lambda}(y))=\lambda^{\nu-q}\,\Gamma(x,y)
\]
which proves the joint homogeneity of $\Gamma$. \medskip

-\,\,\emph{Proof of (III)}. Let $y\in\mathbb{R}^{n}$ be fixed. By the change
of variable $\xi= \Psi_{x,y}^{-1}(\zeta)$, for every $x\neq y$ we can write
\begin{align*}
\Gamma(x,y)  &  = \int_{\mathbb{R}^{p}}\widetilde{\Gamma}_{0} ((y,0)^{-1}%
\star(x,\xi))\,d\xi= \int_{\mathbb{R}^{p}}\widetilde{\Gamma}_{0}
((y,0)^{-1}\star(x,\Psi^{-1}_{x,y}(\zeta)))\,d\zeta\\
&  = \int_{\{\varrho_{E}\leq1\}}\{\cdots\}\,d\zeta+ \int_{\{\varrho_{E} >
1\}}\{\cdots\}\,d\zeta\equiv\mathrm{I}_{1}+\mathrm{I}_{2}.
\end{align*}
Moreover, owing to \eqref{eq:vanishingGammazeroInf} in Theorem
\ref{thm:PropertiesGammGroups}, we easily see that
\[
\lim_{|x|\to+\infty}\widetilde{\Gamma}_{0} ((y,0)^{-1}\star(x,\Psi^{-1}%
_{x,y}(\zeta))) = 0\quad\text{uniformly w.r.t.\,$\zeta\in\mathbb{R}^{p}$}.
\]
The desired assertion \emph{(III)} then follows by a dominated-convergence
argument, cho\-osing as a \emph{dominating function} a constant function for
$\mathrm{I}_{1}$, and $\varrho_{E}^{\nu-Q}$ for $\mathrm{I}_{2}$ (reasoning
like in \ref{eq:StimaChiaveP1}).\medskip

-\,\,\emph{Proof of (IV)}. Assume that $\Gamma_{1},\Gamma_{2}$ are two
functions satisfying properties (b) and \emph{(III)}, and let $y\in
\mathbb{R}^{n}$ be fixed. Setting $\gamma= (\Gamma_{1}-\Gamma_{2})(\cdot,y)$,
by (b) we have
\[
\mathcal{L}\gamma= 0\quad\text{in $\mathbb{R}^{n}$}.
\]
From this, by Theorem \ref{Thm main hypoellipticity}, $\gamma\in C^{\infty
}(\mathbb{R}^{n})$; moreover, by \emph{(III)}, $\gamma$ vanishes at infinity.
We can then apply Theorem \ref{Thm main Liouville}, ensuring that
$\gamma\equiv0$, and uniqueness follows. \medskip

-\,\,\emph{Proof of (V)}. Let $x\neq y\in\mathbb{R}^{n}$ be fixed. By applying
the change of variable
\[
\xi=\Phi_{x,y}(\zeta)
\]
described in Remark \ref{rem:CambiDiVariabile} to the \emph{representation
formula} \eqref{eq:defGammaLiftingRF}, and taking into account Theorem
\ref{thm:PropertiesGammGroups}\thinspace-\thinspace(6), we get the following
computation
\begin{align*}
\Gamma^{\ast}(x,y)  &  =\int_{\mathbb{R}^{p}}\widetilde{\Gamma}_{0}^{\ast
}((y,0)^{-1}\star(x,\xi))\,d\xi\\
&  =\int_{\mathbb{R}^{p}}\widetilde{\Gamma}_{0}^{\ast}\big((y,0)^{-1}%
\star(x,\Phi_{x,y}(\zeta))\big)\,d\zeta\\
&  (\text{by definition of $\Phi_{x,y}$, see \eqref{eq:CVPhiIdentity}})\\
&  =\int_{\mathbb{R}^{p}}\widetilde{\Gamma}_{0}^{\ast}((y,\zeta)^{-1}%
\star(x,0))\,d\zeta\\
&  (\text{since $\widetilde{\Gamma}_{0}^{\ast}(z)=\widetilde{\Gamma}%
_{0}(z^{-1})$, see Theorem \ref{thm:PropertiesGammGroups}})\\
&  =\int_{\mathbb{R}^{p}}\widetilde{\Gamma}_{0}((x,0)^{-1}\star(y,\zeta
))\,d\zeta=\Gamma(y,x),
\end{align*}
which is exactly the desired \eqref{eq:GammaGammastarGC}.
\end{proof}

\medskip

Finally, we can turn to the proof of the \emph{sharp pointwise estimates} for
the global fundamental solution $\Gamma$ of $\mathcal{L}$ and for its
derivatives, stated in Theorem \ref{Thm estimates on Gamma}. To this aim, we
first establish the following \emph{representation formulas} for any
$X$\--derivative of $\Gamma$ (both with respect to $x$ and $y$).

\begin{theorem}
[Formulas for the $X$-derivatives of $\Gamma$]%
\label{thm:ReprFormulaDerivatives} Let $\Gamma$ and $\Gamma^{\ast}$ be as in
Theorem \ref{thm:ExistenceEstimLhGeneral} \emph{(}so that $\Gamma,\Gamma
^{\ast}\in C^{\infty}(\mathbb{R}^{2n}\setminus\mathbb{D})$\emph{)}.

Then, for any $s,t\geq1$, and any choice of indexes
\[
i_{1},\ldots,i_{s},j_{1},\ldots,j_{t}\in\{1,\ldots,m\},
\]
the following representation formulas hold true for $x\neq y$ in
$\mathbb{R}^{n}$:
\begin{align*}
&  X_{i_{1}}^{x}\cdots X_{i_{s}}^{x}\big(\Gamma(\cdot,y)\big)(x)=\int%
_{\mathbb{R}^{p}}\Big(\widetilde{X}_{i_{1}}\cdots\widetilde{X}_{i_{s}%
}\widetilde{\Gamma}_{0}\Big)\Big((y,0)^{-1}\ast(x,\xi)\Big)d\xi;\\[0.2cm]
&  X_{j_{1}}^{y}\cdots X_{j_{t}}^{y}\big(\Gamma(x,\cdot)\big)(y)=\int%
_{\mathbb{R}^{p}}\Big(\widetilde{X}_{j_{1}}\cdots\widetilde{X}_{j_{t}%
}\widetilde{\Gamma}_{0}^{\ast}\Big)\Big((x,0)^{-1}\ast(y,\xi)\Big)\,d\xi
;\\[0.2cm]
&  X_{j_{1}}^{y}\cdots X_{j_{t}}^{y}X_{i_{1}}^{x}\cdots X_{i_{s}}^{x}%
\Gamma(x,y)\\
&  \qquad\quad=\int_{\mathbb{R}^{p}}\bigg(\widetilde{X}_{j_{1}}\cdots
\widetilde{X}_{j_{t}}\Big(\big(\widetilde{X}_{i_{1}}\cdots\widetilde{X}%
_{i_{s}}\widetilde{\Gamma}_{0}\big)\circ\iota\Big)\bigg)\Big((x,0)^{-1}%
\ast(y,\xi)\Big)d\xi.
\end{align*}
Here, $\iota$ denotes the inversion map of the Lie group $\mathbb{G}$.
\end{theorem}

\begin{proof}
The proof of this result is analogous to that of \cite[Thm.\,9.4]%
{BBB-fundsol}. Indeed, it is sufficient to combine the \emph{representation
formula} \eqref{eq:defGammaLiftingRF} (together with identity
(\ref{eq:GammaGammastarGC}) and the change of variable $\xi=\Phi_{x,y}(\zeta)$
from Remark \ref{rem:CambiDiVariabile}) with the following general fact
established in \cite[Lem.\,4.3]{BBB-fundsol}: \vspace{0.05cm}

\emph{If $h\in C^{\infty}(\mathbb{R}^{N}\setminus\{0\})$ is $D_{\lambda}%
$-homogeneous of degree $d<q-Q$ \emph{(}recall that $N>n$ is the topological
dimension of $\mathbb{G}$, $Q$ is its homogeneous dimension and $D_{\lambda}$
is as in \eqref{eq:Dlambdalift}\emph{)}, then for every $x\neq y\in
\mathbb{R}^{n}$ and every $1\leq i\leq m$ we have}
\[
\int_{\mathbb{R}^{p}}X_{i}\bigg\{x\mapsto h\Big((y,0)^{-1}\ast(x,\xi
)\Big)\bigg\}\,d\xi=\int_{\mathbb{R}^{p}}\big(\widetilde{X}_{i}%
h\big)\Big((y,0)^{-1}\ast(x,\xi)\Big)\,d\xi
\]
(\emph{here, $\widetilde{X}_{i}$ is the lifting of $X_{i}$ given in Theorem
\ref{thm:Adrift}}).
\end{proof}

\medskip

\begin{corollary}
For every integer $r\geq0$ there exists $c>0$ such that, for every
$x,y\in\mathbb{R}^{n}$ \emph{(}with $x\neq y$\emph{)}, one has
\begin{equation}
\left\vert Z_{1}\cdots Z_{h }\Gamma(x;y)\right\vert \leq c\int_{\mathbb{R}%
^{p}}d_{\widetilde{X}}^{\nu-Q-r}\left(  (x,0)^{-1}\ast(y,\eta)\right)  d\eta,
\label{stimadel_THEO0}%
\end{equation}
for any choice of $Z_{1},\ldots,Z_{h}$ \emph{(}with $h\leq r$\emph{)} in
\[
\left\{  X_{1}^{x},X_{2}^{x},...,X_{m}^{x},X_{1}^{y},X_{2}^{y},...,X_{m}%
^{y}\right\}  ,
\]
satisfying $\sum_{i = 1}^{h} |Z_{i}| = r$ \emph{(}with $|Z_{i}| = \nu_{k_{i}}$
if $Z_{i} = X_{k_{i}}$ for some $1\leq k_{i}\leq m$\emph{)}. As in the
statement of Theorem \ref{Thm estimates on Gamma}, we understand that
\[
Z_{1}\cdots Z_{h}\Gamma(x,y) = \Gamma(x,y)\quad\text{if $r = 0$}.
\]
In particular, for every $x,y\in\mathbb{R}^{n}$ \emph{(}with $x\neq y$\emph{)}
and $r\geq0$, the function%
\[
\eta\mapsto d_{\widetilde{X}}^{\nu-Q-r}\left(  (x,0)^{-1}\ast(y,\eta)\right)
\]
belongs to $L^{1}\left(  \mathbb{R}^{p}\right)  .$
\end{corollary}

\begin{proof}
Inequality \eqref{stimadel_THEO0} can be proved by arguing as in
\cite[Prop.\,4.6]{BBB-fundsol} (see also \cite[Sec.\,9]{BBB-fundsol}), using
Theorem \ref{thm:ReprFormulaDerivatives} and taking into account the following
facts: \medskip

i)\,\,$\widetilde{\Gamma}_{0}$ is $D_{\lambda}$-homogeneous of degree $\nu-Q$;
\vspace{0.1cm}

ii)\,\,each $\widetilde{X}_{i}$ is $D_{\lambda}$-homogeneous of degree
$\nu_{i}$ \medskip

\noindent The last assertion has been proved in Proposition \ref{prop:P1P2}
(namely, this follows from (\ref{eq:ToProveP1})), keeping into account the
equivalence of $d_{\widetilde{X}}$ and $\varrho_{D}$.
\end{proof}

\medskip

Recall that, in order to build a fundamental solution for $\mathcal{L}$, we
have assumed%
\[
\nu<q\left(  <Q\right)  .
\]
In order to bound the integral in (\ref{stimadel_THEO0}) with a more
transparent quantity, arriving at the estimates contained in Theorem
\ref{Thm estimates on Gamma}, we now need to strengthen our assumption on
$\nu$, \emph{or }assume $r$ is large enough: see condition (\ref{vincolo su r}%
). We can now give the

\bigskip

\begin{proof}
[Proof of Theorem \ref{Thm estimates on Gamma}]Point (1) follows from
(\ref{stimadel_THEO0}) following the reasoning in \cite[Proof of Prop.\,4.7
and Prop.\,4.8]{BBB-fundsol}, while Point (2) can be proved as in
\cite[Sec.\,6]{BBB-fundsol}.
\end{proof}

\subsection{Extension to heat-type operators\label{sec heat type operators}}

Here we briefly discuss some extensions of our theory to heat-type operators
of the kind%
\[
\mathcal{H}=\mathcal{L}\pm\partial_{t}%
\]
defined on $\mathbb{R}^{n+1}\ni\left(  x,t\right)  $ where $\mathcal{L}$
(acting on the $x$ variable) is a generalized Rockland operator on
$\mathbb{R}^{n}$. If $\mathcal{L}$ is $\nu$-homogeneous w.r.t.\,dilations
$\left\{  \delta_{\lambda}\right\}  _{\lambda}$ in $\mathbb{R}^{n}$, then
$\mathcal{H}$ will be $\nu$-homogeneous w.r.t.\,dilations%
\begin{equation}
\delta_{\lambda}^{\prime}\left(  x,t\right)  =\left(  \delta_{\lambda}\left(
x\right)  ,\lambda^{\nu}t\right)  \label{dilations delta primo}%
\end{equation}
in $\mathbb{R}^{n+1}$. If $\widetilde{\mathcal{L}}$ is the Rockland operator
obtained by lifting $\mathcal{L}$ to the homogeneous group $\mathbb{G}=\left(
\mathbb{R}^{N},\star,\left\{  D_{\lambda}\right\} _{\lambda}\right)  $ (see
Theorem \ref{thm:Adrift}), then%
\[
\widetilde{\mathcal{H}}=\widetilde{\mathcal{L}}\pm\partial_{t}%
\]
will be a lifting of $\mathcal{H}$ on the homogeneous group
\[
\mathbb{G}^{\prime}=(\mathbb{R}^{N+1},\star^{\prime},\left\{  D_{\lambda
}^{\prime}\right\}  _{\lambda}),
\]
where (setting $u=\left(  u^{\prime},u_{N+1}\right)  ,v=\left(  v^{\prime
},v_{N+1}\right)  \in\mathbb{R}^{N+1}$),
\begin{align*}
&  u\star^{\prime}v=\left(  u^{\prime}\star v^{\prime},u_{N+1}+v_{N+1}\right)
;\\
&  D_{\lambda}^{\prime}\left(  u^{\prime},u_{N+1}\right)  =\left(  D_{\lambda
}\left(  u^{\prime}\right)  ,\lambda^{\nu}u_{N+1}\right)  .
\end{align*}
Let us now suppose that, moreover, $\widetilde{\mathcal{L}}$ is a
\emph{positive }Rockland operator, in the sense of \cite[Sec.\,4.2]{FR}, that
is:
\begin{equation}%
\begin{tabular}
[c]{c}%
$\widetilde{\mathcal{L}}$ is formally self-adjoint and satisfies\\[0.2cm]%
$\int_{\mathbb{R}^{N}}\widetilde{\mathcal{L}}f(x)\cdot\overline{f\left(
x\right)  }dx\geq0\text{ for every }f\in C_{0}^{\infty}\left(  \mathbb{R}%
^{N}\right)  .$%
\end{tabular}
\ \label{positive Rockland}%
\end{equation}
Then, by \cite[Lem.\,4.2.11]{FR} the (lifted) operators
$\widetilde{\mathcal{H}}=\widetilde{\mathcal{L}}\pm\partial_{t}$ are Rockland
operators on $\mathbb{G}^{\prime}$, hence $\mathcal{H}=\mathcal{L}\pm
\partial_{t}$ are generalized Rockland operators on $\mathbb{R}^{n+1}$,
homogeneous of degree $\nu$. We can therefore apply to the heat-type operators
$\mathcal{L}\pm\partial_{t}$ our previous theory. If $q$ is the homogeneous
dimension of $\mathbb{R}^{n}$ w.r.t.\,the dilations $\left\{  \delta_{\lambda
}\right\}  $, then $q+\nu$ will be the homogeneous dimension of $\mathbb{R}%
^{n+1}$ w.r.t.\,the dilations
\[
\delta_{\lambda}^{\prime}\left(  x,t\right)  =\left(  \delta_{\lambda}\left(
x\right)  ,\lambda^{\nu}t\right)  .
\]
Note that the condition $\nu<q$ in Theorem \ref{thm:ExistenceEstimLhGeneral}
and $r\geq\nu-n$ in Theorem \ref{Thm estimates on Gamma} now read,
respectively, $\nu<q+\nu$ (which is \emph{always} satisfied) and $r\geq
\nu-n-1$, which is \emph{weaker} than $r\geq\nu-n$.

Before stating our conclusion, let us give the following:

\begin{definition}
Let $\mathcal{L}$ be a generalized Rockland operator. We say that
$\mathcal{L}$ is a \emph{positive }generalized Rockland operator if its lifted
operator $\widetilde{\mathcal{L}}$ (in the sense of Section \ref{sec lifting})
is a positive\emph{ }Rockland operator in the sense of
(\ref{positive Rockland}).
\end{definition}

Then, with the above discussion we have proved the following:

\begin{theorem}
Let%
\[
\mathcal{L}=\sum_{|I|=\nu}c_{I}X_{I}%
\]
be a positive generalized Rockland operator on $\mathbb{R}^{n}$, with $\nu$
the homogeneity degree of $\mathcal{L}$ and $q$ the homogeneous dimension of
$\mathbb{R}^{n}$. Let us consider the heat-type operator%
\[
\mathcal{H}=\mathcal{L}\pm\partial_{t},
\]
which is $\nu$-homogeneous on $\mathbb{R}^{n+1}$, w.r.t.\,the dilations
(\ref{dilations delta primo}). Then:

\begin{itemize}
\item[i)] the operator $\mathcal{H}$ is a generalized Rockland operator, in
particular it is hypoelliptic in $\mathbb{R}^{n+1}$ and satisfies a Liouville
property as in Theorem \ref{Thm main Liouville}.

\item[ii)] the operator $\mathcal{H}$ possesses a global fundamental solution
$\Gamma((x,t),(y,s)) $ join\-tly $\delta_{\lambda}^{\prime}$-homogeneous of
degree $-q$ and satisfying the other properties stated in Theorem
\ref{thm:ExistenceEstimLhGeneral}.

\item[iii)] let $r$ be a nonnegative integer such that
\[
r\geq\nu-n-1.
\]
Then $\Gamma\left(  \left(  x,t\right)  ,\left(  y,s\right)  \right)  $
satisfies the pointwise bounds expressed by Theorem
\ref{Thm estimates on Gamma} with
\[%
\begin{tabular}
[c]{c}%
$q$\\
$\nu$\\
$n$%
\end{tabular}
\text{ in Theorem \ref{Thm estimates on Gamma} replaced by }%
\begin{tabular}
[c]{c}%
$q+\nu$\\
$\nu$\\
$n+1$%
\end{tabular}
,
\]
where now $Z_{1},\ldots,Z_{h}$ are chosen in%
\[
\left\{  X_{1}^{x},X_{2}^{x},...,X_{m}^{x},\partial_{t},X_{1}^{y},X_{2}%
^{y},...,X_{m}^{y},\partial_{s}\right\}  ,
\]
the vector fields $\partial_{t},\partial_{s}$ have \emph{weight }$\nu$ and, as
in Theorem \ref{Thm estimates on Gamma}, the number $r=\sum_{i=1}%
^{h}\left\vert Z_{i}\right\vert $ is the total weight of the differential
operator $Z_{1}\cdots Z_{h}$.
\end{itemize}
\end{theorem}

Let us specialize the previous discussion to the case of operators (\ref{Rock}).

\begin{theorem}
\label{Thm heat operator}Let $X_{1},...,X_{m}$ be a family of vector fields in
$\mathbb{R}^{n}$ satisfying Assumption (H1) and (with the same notation of
Assumption (H1)) for any positive integer $\nu_{0}$, common multiple of
$\nu_{1},\nu_{2},...,\nu_{m}$, let us consider the operators, defined in
$\mathbb{R}^{n+1}\ni\left(  x,t\right)  ,$
\begin{equation}
\mathcal{H}=\sum_{j=1}^{m}\left(  -1\right)  ^{\frac{\nu_{0}}{\nu_{j}}}%
X_{j}^{\frac{2\nu_{0}}{\nu_{j}}}\pm\partial_{t}, \label{heat type}%
\end{equation}
which is $2\nu_{0}$-homogeneous in $\mathbb{R}^{n+1}$ w.r.t.\,the dilations%
\[
\delta_{\lambda}^{\prime}\left(  x,t\right)  =\left(  \delta_{\lambda}\left(
x\right)  ,\lambda^{2\nu_{0}}t\right)  .
\]
Then:

\begin{itemize}
\item[i)] the operator $\mathcal{H}$ is hypoelliptic in $\mathbb{R}^{n+1}$ and
satisfies the Liouville property expressed in Theorem \ref{Thm main Liouville}.

\item[ii)] the operator $\mathcal{H}$ possesses a global fundamental solution
$\Gamma\left(  \left(  x,t\right)  ,\left(  y,s\right)  \right)  $ join\-tly
$\delta_{\lambda}^{\prime}$-homogeneous of degree $-q$ and satisfying the
other properties stated in Theorem \ref{thm:ExistenceEstimLhGeneral}.

\item[iii)] Let $r$ be a nonnegative integer such that
\[
r\geq2\nu_{0}-n-1.
\]
Then $\Gamma\left(  \left(  x,t\right)  ,\left(  y,s\right)  \right)  $
satisfies the pointwise bounds expressed by Theorem
\ref{Thm estimates on Gamma} with%
\begin{equation}%
\begin{tabular}
[c]{c}%
$q$\\
$\nu$\\
$n$%
\end{tabular}
\text{ in Theorem \ref{Thm estimates on Gamma} replaced by }%
\begin{tabular}
[c]{c}%
$q+2\nu_{0}$\\
$2\nu_{0}$\\
$n+1$%
\end{tabular}
\label{changed parameters}%
\end{equation}
where $Z_{1},\ldots,Z_{h}$ are chosen in%
\[
\left\{  X_{1}^{x},X_{2}^{x},...,X_{m}^{x},\partial_{t},X_{1}^{y},X_{2}%
^{y},...,X_{m}^{y},\partial_{s}\right\}  ,
\]
the vector fields $\partial_{t},\partial_{s}$ have \emph{weight }$2\nu_{0}$
and $r=\sum_{i=1}^{h}\left\vert Z_{i}\right\vert $ is the total weight of the
differential operator $Z_{1}\cdots Z_{h}$.
\end{itemize}
\end{theorem}

\begin{proof}
It is enough to check that%
\[
\mathcal{L}=\sum_{j=1}^{m}\left(  -1\right)  ^{\frac{\nu_{0}}{\nu_{j}}}%
X_{j}^{\frac{2\nu_{0}}{\nu_{j}}}%
\]
is a \emph{positive }Rockland operator. Actually, denoting by
\[
\widetilde{\mathcal{L}} = \sum_{j=1}^{m}\left(  -1\right)  ^{\frac{\nu_{0}%
}{\nu_{j}}}\widetilde{X}_{j}^{\frac{2\nu_{0}}{\nu_{j}}}%
\]
the lifted operator of $\mathcal{L}$, we see that $\widetilde{\mathcal{L}}$ is
formally self-adjoint and%
\begin{align*}
\int_{\mathbb{R}^{N}} \widetilde{\mathcal{L}}f\left(  x\right)  \cdot
\overline{f\left(  x\right)  }dx  &  =\sum_{j=1}^{m}\int_{\mathbb{R}^{N}%
}\left(  -1\right)  ^{\frac{\nu_{0}}{\nu_{j}}}\left(  \widetilde{X}_{j}%
^{\frac{2\nu_{0}}{\nu_{j}}}f\right)  \left(  x\right)  \overline{f\left(
x\right)  }dx\\
&  =\sum_{j=1}^{m}\int_{\mathbb{R}^{N}}\left\vert \widetilde{X}_{j}^{\frac
{\nu_{0}}{\nu_{j}}}f\right\vert ^{2}\left(  x\right)  dx\geq0.
\end{align*}
Then the assumptions of Theorem \ref{Thm heat operator} are satisfied.
\end{proof}

\bigskip

\noindent\textbf{Acknowledgements.} We warmly thank the anonymous Referees for
their careful reading of the manu\-script and for their valuable comments,
which lead us to improve the paper.

The Authors are members of the research group \textquotedblleft Grup\-po
Na\-zio\-nale per l'Analisi Matematica, la Probabilit\`{a} e le loro
Applicazioni\textquotedblright\ of the Italian \textquotedblleft Istituto
Na\-zio\-na\-le di Alta Matematica\textquotedblright. The first Author is
partially supported by the PRIN 2022 project 2022R537CS \emph{$NO^{3}$ - Nodal
Optimization, NOnlinear elliptic equations, NOnlocal geometric problems, with
a focus on regularity}, funded by the European Union - Next Generation EU; the
second Author is partially supported by the PRIN 2022 project \emph{Partial
differential equations and related geometric-functional inequalities},
financially supported by the EU, in the framework of the \textquotedblleft
Next Generation EU initiative\textquotedblright.

\end{document}